\documentclass[reqno,10pt, centertags]{amsart}
\usepackage{amssymb,upref,esint,color}
\usepackage{hyperref}
\usepackage{etoolbox}
\newcommand*{\mailto}[1]{\href{mailto:#1}{\nolinkurl{#1}}}

\makeatletter
\usepackage{verbatim}
\usepackage{amscd}
\usepackage{framed}
\numberwithin{equation}{section}
\input epsf

\newtheorem{theorem}{Theorem}[section]
\newtheorem{definition}[theorem]{Definition}
\newtheorem{lemma}[theorem]{Lemma}
\newtheorem{corollary}[theorem]{Corollary}
\newtheorem{proposition}[theorem]{Proposition}
\numberwithin{equation}{section}
\theoremstyle{definition}
\newtheorem{example}[theorem]{Example}

\begin{document}

\title[A new functional model for contractions]
{A new functional model for contractions}

\author[Yicao Wang]{Yicao Wang}
\address
{School of Mathematics, Hohai University, Nanjing 210098, China}

\baselineskip= 20pt
\begin{abstract}The paper presents a new functional model for completely non-unitary contractions on a Hilbert space. This model is based on the observation that the theory of contractions shares a common geometric basis with the extension theory of symmetric operators recently developed by the author in \cite{wang2024complex}. Compared with the now classical Sz.-Nagy-Foias model and the de Branges-Rovnyak model, ours is intrinsic in the sense that we need not construct a bigger space $H$ including the model space $\mathfrak{H}$ and realize the model operator on $\mathfrak{H}$ as the compression of a minimal unitary dilation on $H$. Our model space $\mathfrak{H}$ is constructed in a canonical and conceptually more direct manner and doesn't depend on the Sz. Nagy-Foias characteristic function explicitly. We also show how a contraction can be constructed from a marked Nevanlinna disc, which is the geometric analogue of the characteristic function.
\end{abstract}
\maketitle


\tableofcontents

\section{Introduction}
In  \cite{wang2024complex}, the author has proposed a complex geometric formalism to treat von Neumann's theory of self-adjoint extensions of symmetric operators. The theory is based on the geometry of Hermitian symmetric spaces of type $I_{p,q}$ and Nevanlinna curves in such spaces. In particular, the notion of Weyl curve $W_T(\lambda)$ associated to a simple symmetric operator $T$ played a central role there. $W_T(\lambda)$ is a Nevanlinna curve, containing all the unitarily invariant information of $T$. Conversely, given a Nevanlinna curve $N(\lambda)$, we can construct a simple symmetric operator $\mathfrak{T}$ acting in a reproducing kernel Hilbert space $\mathfrak{H}$ of holomorphic sections of a certain vector bundle over the upper and lower planes $\mathbb{C}_+\cup \mathbb{C}_-$. If $N(\lambda)$ is just the Weyl curve $W_T(\lambda)$, then the corresponding $\mathfrak{T}$ is unitarily equivalent to $T$. In this sense, $\mathfrak{T}$ is a \emph{model operator} of $T$ and $\mathfrak{H}$ is called the \emph{model space}.

A remarkable point is that, if a point $p$ in the non-compact Hermitian symmetric space where $W_T(\lambda)$ lives is given, then it can be used to introduce a coordinate system so that $W_T(\lambda)$ can be represented by an analytic contraction-valued (i.e., Schur) function defined on $\mathbb{C}_+$. Schur functions also appear in many other different areas, e.g., Pick interpolation problems, system theory and prediction theory. What interests us here is the theory of contractions on Hilbert spaces. In the classical Sz.-Nagy-Foias theory on contractions, a central object is the characteristic function $\Theta_T(\lambda)$ of a contraction $T$, which is the main component in constructing a functional model for $T$. $\Theta_T(\lambda)$ is again an operator-valued Schur function, though its domain is the unit disc $\mathbb{D}$ rather than the upper half plane $\mathbb{C}_+$. There is an alternative competitive model theory for contractions suggested by de Branges and Rovnyak \cite{de2015square}, which is also based on Schur functions. In this latter model, the model space is a reproducing kernel Hilbert space consisting of certain vector-valued holomorphic functions on $\mathbb{D}$, whose reproducing kernel is generated by the Schur function in use.

The similarity between the theory of contractions and that of symmetric operators might not be something new and even be sometimes obvious. However, as far as we know, there seems to be no clear explanation on this in the existing literature. One of the goals of the paper is thus to fill in this gap. We find that these two theories actually share the same geometric basis, which involves strong symplectic Hilbert spaces, their relevant Hermitian symmetric spaces, their isotropic subspaces and a related quotient construction. Without any essential modification, the way of defining Weyl curve in \cite{wang2024complex} can be applied to contractions as well. Then $\Theta_T(\lambda)$ turns out to be a specific representation of the underlying Weyl curve (see Prop.~\ref{p3}).

Now that the two theories share the common basis, can the methodology of \cite{wang2024complex} be used to produce a functional model for contractions? The second goal of the paper is to provide an affirmative answer to this question. In either the Sz.-Nagy-Foias model or the de Branges-Rovnyak model, given the characteristic function $\Theta_T(\lambda)$, the basic idea is to construct a minimal unitary dilation $\mathbb{U}$ acting on a much bigger Hilbert space $H$ and to realize the model operator $\mathfrak{T}$ as the compression of $\mathbb{U}$ to the model space $\mathfrak{H}\subseteq H$. In sharp contrast, our approach doesn't adopt this embedding viewpoint and can be viewed as being intrinsic in the sense that it doesn't involve a construction of $H$ and a unitary dilation on it, and that the role of $\Theta_T(\lambda)$ is only secondary. Our model space $\mathfrak{H}$ is a reproducing kernel Hilbert space consisting of certain holomorphic sections of a characteristic vector bundle over two copies of the unit disc $\mathbb{D}$. When a certain trivialization of the vector bundle is used and $\mathfrak{H}$ is interpreted as a space of vector-valued analytic functions, our result is much closer to the de Branges-Rovnyak model.

Our new model theory cannot be regarded as a surpassing of its two deep-reaching precursors, but it really sheds some new light on the model theories: A contraction $T$ contains two degrees of freedom, its unitary map part and its non-unitary map part. The model space $\mathfrak{H}$ only depends on the former and hence contractions with the same unitary map part all share the same model space. Besides, we take the viewpoint that to construct a model for $T$ is the same thing as to construct a model for the graph of $T$, which is a subspace of the symplectic orthogonal complement $A_T^{\bot_s}$ of the underlying isotropic subspace $A_T$ induced from the unitary map part of $T$. Then constructing a model for $T$ amounts to constructing a model for $A_T^{\bot_s}$ and putting on it an abstract "boundary condition" induced from the non-unitary map part of $T$ to cut the graph of the model operator $\mathfrak{T}$ out. This viewpoint is much more flexible because many operators sharing the same unitary map part as $T$ can be realized altogether in this universal manner.

The paper is organized as follows. In \S~\ref{sec2}, we collect the basics on strong symplectic Hilbert spaces and the corresponding non-compact Hermitian symmetric spaces. The material can mostly be found in \cite[\S~3]{wang2024complex}, but that on maximal positive-definite subspaces is motivated by the theory of contractions. \S~\ref{sec3} is devoted to outlining the geometric formalism developed in \cite[\S~4-5]{wang2024complex} to deal with symmetric operators. By comparison with \S~\ref{sec4}, the reader will find how the theory on symmetric operators and that on contractions are related to each other. Roughly speaking, they are connected by the Cayley transform. In \S~\ref{sec4.1}, we recall some basic facts on contractions and show how they can be viewed from the angle established in \S~\ref{sec2}. We can then define the Weyl curve $W_T(\lambda)$ of a contraction $T$. We haven't included proofs of the basic properties of $W_T(\lambda)$ since these would be clear as soon as one comes to realize the relation between the theory of symmetric operators and that of contractions. In \S~\ref{sec4.2}, we show how a canonical model of a completely non-unitary contraction $T$ can be constructed in terms of natural holomorphic vector bundles associated to $T$. A Nevanlinna disc together with a maximal positive-definite subspace is called a marked Nevanlinna disc in \S~\ref{sec5}. We show there how one can construct a contraction from a marked Nevanlinna disc. This corresponds to the well-known fact that from a Schur function on the unit disc a contraction can be constructed. Some of the relevant proofs are only sketched for the details can be obtained by modifying those in \cite[\S~5.2]{wang2024complex} slightly.
\begin{center}
\textbf{Conventions and notations}
\end{center}

 In this paper, all (finite-dimensional or infinite-dimensional) Hilbert spaces are complex and separable. We adopt the convention that the inner product $(\cdot, \cdot)$ is linear in the first variable and conjugate-linear in the second. The induced norm will be written as $\|\cdot\|$. If necessary, the inner product or norm is also denoted by $(\cdot, \cdot)_H$ or $\|\cdot\|_H$ to emphasize the underlying Hilbert space. $(\cdot, \cdot)$ is also used to denote an ordered pair and the context shall exclude the possible ambiguities. We use $\oplus$ to denote topological direct sums while $\oplus_\bot$ is used to denote orthogonal direct sums. For a subspace $V$, its orthogonal complement is $V^\bot$. The zero vector space will be just denoted by $0$. $\mathbb{B}(H_1, H_2)$ denotes the Banach space of bounded operators between the Hilbert spaces $H_1$ and $H_2$. If $H_1=H_2=H$, this would be simplified as $\mathbb{B}(H)$. The identity operator on $H$ will be just denoted by $Id$ if the underlying space is clear from the context. For a constant $c$, $c\times Id$ is often written simply as $c$ if no confusion arises. $D(A)$ is the domain of the operator $A$, and $\textup{ker}A$ (resp. $\textup{Ran}A$) is the kernel (resp. range) of $A$. For operators $A$ and $B$, we use $A\subset B$ to mean
$D(A)\subset D(B)$ and $B|_{D(A)}=A$. In this setting, $B$ is called an extension of $A$.
\section{Strong symplectic Hilbert spaces and related Hermitian symmetric spaces}\label{sec2}
In this section, we collect the basics concerning strong symplectic Hilbert spaces and the non-compact Hermitian symmetric spaces associated to them. This material lays the foundation for later sections. The presentation is sketchy and we refer the reader to \cite{wang2024complex} for most technical details.

Let $H$ be a complex Hilbert space whose dimension $\textup{dim}H\leq +\infty$.
\begin{definition}
 A strong symplectic structure on $H$ is a continuous sesquilinear form $[\cdot, \cdot]: H\times H\rightarrow \mathbb{C}$ such that\\
 i) For $x,y\in H$, $[x, y]$ is linear in $x$ and conjugate-linear in $y$;\\
 ii) $[y,x]=-\overline{[x,y]}$ for all $x,y\in H$;\\
 iii)  $[\cdot, \cdot]$ is non-degenerate in the sense that if $[x,y]=0$ for all $y\in H$, then $x=0$;\\
 iv) the map $\mathfrak{i}:H\rightarrow H^*$ defined by $\mathfrak{i}(x)=[\cdot, x]$ is surjective, where $H^*$ is the dual space of continuous linear functionals on $H$.\\
 The pair $(H, [\cdot, \cdot])$ will be called a strong symplectic Hilbert space and if no ambiguity arises, we just say $H$ is a strong symplectic Hilbert space.
  \end{definition}
 Note that for a strong symplectic structure $[\cdot, \cdot]$, $-i[\cdot, \cdot]$ is nothing else but a non-degenerate indefinite Hermitian inner product on $H$. We prefer our terminology here because we want to emphasize the similarity to real symplectic spaces. The signature of $-i[\cdot, \cdot]$ will be denoted by $(n_+,n_-)$ and called the signature of $[\cdot, \cdot]$ or $(H, [\cdot, \cdot])$. We mainly consider the case that both $n_\pm$ are nonzero.
\begin{example}\label{e3} Given two Hilbert spaces $H_+$, $H_-$, of dimension $n_+$ and $n_-$ respectively, then $H=H_+\oplus_\bot H_-$ can be equipped with a standard strong symplectic structure as follows: For $x=(x_1,x_2), y=(y_1,y_2)\in H_+\oplus_\bot H_-$, define
 \[[x,y]_{st}=i(x_1,y_1)_{H_+}-i(x_2,y_2)_{H_-}.\]
 This strong symplectic structure has signature $(n_+,n_-)$.
  \end{example}
  \begin{example}\label{e5}Given a Hilbert space $H$ of dimension $n$, on $\mathbb{H}:=H\oplus_\bot H$, beside the standard strong symplectic structure $[\cdot, \cdot]_{st}$ in Example \ref{e3} with $H_+=H_-=H$, there is an alternative non-standard strong symplectic structure defined via
  \[[(x_1,y_1),(x_2,y_2)]_{new}=(y_1,x_2)_H-(x_1,y_2)_H\]
  where $(x_i,y_i)\in H\oplus_\bot H$, $i=1,2$. This strong symplectic structure has signature $(n,n)$ (see Example \ref{e8}).
  \end{example}

   \begin{example}\label{e10}If $(H_i, [\cdot, \cdot]_i)$, $i=1,2$ are two strong symplectic Hilbert spaces, the direct sum of the two is the orthogonal direct sum $H_1\oplus_\bot H_2$, whose strong symplectic structure is defined as follows: For $(x_1,y_1), y=(x_2,y_2)\in H_1\oplus_\bot H_2$,
  \[[(x_1,y_1),(x_2,y_2)]:=[x_1,x_2]_1+[y_1,y_2]_2.\]
  \end{example}
 Let $H$ be a strong symplectic Hilbert space.
 \begin{definition}For a subspace $L\subseteq H$, the symplectic orthogonal complement $L^{\bot_s}$ is
\[L^{\bot_s}=\{x\in H|[x,y]=0,\ \forall y\in L\}.\]
\end{definition}
It can be found that for a (not necessarily closed) subspace $L$, $(L^{\bot_s})^{\bot_s}=\overline{L}$.
\begin{definition}A closed subspace $L\subseteq H$ is isotropic if $[x,y]=0$ for all $x,y\in L$.
\end{definition}
By definition, for an isotropic subspace $L$, $L\subseteq L^{\bot_s}$ and the quotient $L^{\bot_s}/L$ makes sense. If furthermore $L=L^{\bot_s}$, $L$ is called Lagrangian. Lagrangian subspaces only exist when $n_+=n_-$. For us, the following proposition is basic.
\begin{proposition}For an isotropic subspace $L$ of $H$, $L^{\bot_s}/L$ has a canonical structure of strong symplectic Hilbert space.
\end{proposition}
\begin{proof}Since $L^{\bot_s}$ is closed in $H$, $L^{\bot_s}/L$ has a canonical quotient Hilbert space structure. For $[x], [y]\in L^{\bot_s}/L$ where $x,y\in L^{\bot_s}$, the quotient strong symplectic structure is defined by $[[x],[y]]_q:=[x,y]$. It's easy to check this is well-defined. We leave other details to the interested reader.
\end{proof}
In the above sense, we shall say $L^{\bot_s}/L$ has a canonical quotient strong symplectic structure.
\begin{definition}A closed subspace $L\subseteq H$ is called maximally positive-definite in case for each $0\neq x\in L$, $-i[x,x]>0$ and $L$ cannot be contained properly in a subspace fulfilling a similar inequality. $L$ is called maximally completely positive-definite if $L$ is maximally positive-definite and
\[-i[x,x]\geq c_L\|x\|^2,\quad \forall\  x\in L\]
for a constant $c_L>0$. Maximal (completely) negative-definite subspaces of $H$ can be defined analogously.
\end{definition}

If $L$ is maximally completely positive-definite, then $L^{\bot_s}$ is maximally completely negative-definite and $H=L\oplus L^{\bot_s}$. In particular, $\dim L=n_+$  and $\dim L^{\bot_s}=n_-$. $-i[\cdot, \cdot]$ (resp. $i[\cdot, \cdot]$) restricted on $L$ (resp. $L^{\bot_s}$) is a positive-definite inner product and both $L$ and $L^{\bot_s}$ are Hilbert spaces in this sense. These inner products will be denoted by $(\cdot,\cdot)_\pm$. If a maximal completely positive-definite subspace $L\subseteq H$ is specified, we say $H$ is polarized or given a polarization by $L$.
\begin{example}In Example \ref{e3}, $L:=\{(x,0)\in H|x\in H_+\}$ is obviously a maximal completely positive-definite subspace of $H$. We shall call it the standard polarization of $H$. In Example \ref{e5}, with the decomposition $\mathbb{H}=H\oplus_\bot H$ in mind, we say $(\mathbb{H}, [\cdot, \cdot]_{new})$ is equipped with its non-standard polarization although neither copy of $H$ is maximally completely positive-definite.
\end{example}

Two strong sympletic Hilbert spaces $(H_1, [\cdot, \cdot]_1)$ and $(H_2, [\cdot, \cdot]_2)$ are said to be isomorphic, if there is a topological linear isomorphism $\Phi$ from $H_1$ to $H_2$ such that $[\Phi(x),\Phi(y)]_2=[x,y]_1$ for all $x,y\in H_1$. Note that here we \emph{don't} require that $\Phi$ be an isometry between $H_1$ and $H_2$. All strong symplectic Hilbert spaces with the same signature $(n_+,n_-)$ are isomorphic. In particular, all automorphisms of such a strong symplectic Hilbert space form a Banach-Lie group--the pseudo-unitary group $\mathbb{U}(n_+,n_-)$. By definition, an element in $\mathbb{U}(n_+,n_-)$ moves a maximal (completely) positive-definite subspace to another.
\begin{example}\label{e4}Given a strong symplectic Hilbert space $H$ and a polarization $L$, since $H=L\oplus L^{\bot_s}$, each $x\in H$ has an $L$-part $x_+$ and an $L^{\bot_s}$-part $x_-$. Then for $a,b\in H$,
\[[a,b]=i(a_+.b_+)_+-i(a_-,b_-)_-.\]
This of course means $\Phi(x)=(x_+,x_-)$ for all $x\in H$ is a symplectic isomorphism between $H$ and the standard strong symplectic Hilbert space in Example \ref{e3} with $H_+=L$ and $H_-=L^{\bot_s}$.
\end{example}
\begin{example}\label{e8}Given a Hilbert space $H$, on $\mathbb{H}:=H\oplus_\bot H$, we call the transform
\[\beta(x,y)=(\frac{y+ix}{\sqrt{2}},\frac{y-ix}{\sqrt{2}}),\quad \forall \ (x,y)\in \mathbb{H}\]
the Cayley transform on $\mathbb{H}$. It can be checked easily that $\beta$ is a symplectic isomorphism between $(\mathbb{H}, [\cdot,\cdot]_{new})$ and $(\mathbb{H}, [\cdot,\cdot]_{st})$.
\end{example}

For a strong symplectic Hilbert space $H$, let $W_\pm(H)$ be the set of all maximal completely positive/negative-definite subspaces. $W_+(H)$ (or $W_-(H)$) has a natural complex Banach manifold structure. If a polarization $L$ for $H$ is given, then any $W\in W_+(H)$ can be parameterized uniquely by an operator $B\in \mathbb{B}(L,L^{\bot_s})$ such that $\|B\|<1$, i.e.,
\[W=\{x+Bx\in H|x\in L\}\]
while
\[W^{\bot_s}=\{B^*x+x\in H|x\in L^{\bot_s}\}.\]
Via the symplectic isomorphism $\Phi$ in Example \ref{e4}, $W$ is turned into $\{(x,Bx)\in L\oplus_\bot L^{\bot_s}|x\in L\}$ while $W^{\bot_s}$ into $\{(B^*x,x)\in L\oplus_\bot L^{\bot_s}|x\in L^{\bot_s}\}$. In other words, a polarization provides a global coordinate chart for $W_+(H)$ or $W_-(H)$.

Given the polarization $L$, it is easy to see that a general maximal positive-definite subspace $W\subseteq H$ can also be represented in the above manner except that now the parameter $B$ is a pure contraction, i.e., for each nonzero $x\in L$, $\|Bx\|<\|x\|$. It is possible that $\|B\|=1$ in this case.
\begin{proposition}Let $W$ be a maximal positive-definite subspace of $H$. If either of $n_\pm$ is finite, then $W$ is also maximally completely positive-definite.
\end{proposition}
\begin{proof}Choose a polarization $L$ for $H$ and let $W$ be parameterized by $B\in \mathbb{B}(L,L^{\bot_s})$. Then $B$ is a pure contraction. Note that
\[\|B\|^2=\sup_{\|x\|=1}(Bx,Bx)_-=\sup_{\|x\|=1}(B^*Bx,x)_+.\]
If either of $n_\pm$ is finite, $B^*B$ is of finite rank and thus a compact self-adjoint operator. If $\|B\|=1$, then the largest eigenvalue of $B^*B$ is 1. Let $x$ be a corresponding eigenvector. Then we should have $\|Bx\|=\|x\|$. This is a contradiction!
\end{proof}
Thus the distinction between a maximal positive-definite subspace and a maximal completely positive-definite subspace only occurs when both $n_\pm$ are infinite. The subtlety of a maximal positive-definite subspace is indicated by the following proposition.
\begin{proposition}\label{p4}If $H$ is a strong symplectic Hilbert space with signature $(\infty, \infty)$ and $W$ a maximal positive-definite subspace that is not maximally completely positive-definite, then $W\cap W^{\bot_s}=0$ and $W\oplus W^{\bot_s}$ is a proper dense subspace of $H$.
\end{proposition}
\begin{proof}The first claim is clear. To prove the second, we choose a polarization $L$ for $H$ and let $W$ be parameterized by $B\in \mathbb{B}(L, L^{\bot_s})$. Then we must have $\|B\|=1$. To see if $W\oplus W^{\bot_s}$ is equal to $H$ or not, it is enough to check whether the operator matrix $\mathfrak{B}:=\left(
                                                                                                      \begin{array}{cc}
                                                                                                        Id & B^* \\
                                                                                                        B & Id \\
                                                                                                      \end{array}
                                                                                                    \right)
$ on $L\oplus_\bot L^{\bot_s}$ is surjective or not. Since $B$ is a pure contraction, we can find that 0 is not an eigenvalue of the self-adjoint operator $\mathfrak{B}$. It can also be seen easily that $\mathfrak{B}$ is boundedly invertible if and only if $\|B\|<1$. Thus 0 lies in the continuous spectrum of $\mathfrak{B}$ and $\textup{Ran}\mathfrak{B}$ is a proper dense subspace of $L\oplus_\bot L^{\bot_s}$.
\end{proof}
\emph{Remark}. Thus unlike a maximal completely positive-definite subspace $L$, a maximal positive-definite subspace $W$ may not introduce a decomposition on $H$ and won't provide a coordinate chart for $W_\pm(H)$.

Actually when both $n_\pm$ are finite, $W_+(H)$ is a non-compact irreducible Hermitian symmetric space of type $I_{n_+,n_-}$ and has dimension $n_+\cdot n_-$. Strictly speaking, when either $n_+$ or $n_-$ is infinite, $W_+(H)$ is an infinite-dimensional Banach symmetric space which is much more complicated. We will not distinguish between these two cases and simply call them Hermitian symmetric spaces of type $I_{n_+,n_-}$. A basic fact is that $W_+(H)$ is a homogeneous complex manifold of $\mathbb{U}(n_+,n_-)$. In fact, if $L$ is a chosen polarization for $H$ and $B\in \mathbb{B}(L, L^{\bot_s})$ parameterizes $W\in W_+(H)$, then w.r.t. the identification $H=L\oplus_\bot L^{\bot_s}$, the operator matrix
\[\left(
    \begin{array}{cc}
      (Id-B^*B)^{-1/2} & -B^*(Id-BB^*)^{-1/2} \\
      B(Id-B^*B)^{-1/2} & -(Id-BB^*)^{-1/2} \\
    \end{array}
  \right)
\]
lies in $\mathbb{U}(n_+,n_-)$, transforming $L$ (parameterized by 0) into $W$ (parameterized by $B$).

Let $A$ be an isotropic subspace of $H$. If the signature of the quotient strong symplectic structure on $A^{\bot_s}/A$ is $(n_+,n_-)$ and an isomorphism $\Phi$ between $A^{\bot_s}/A$ and the standard strong symplectic Hilbert space in Example \ref{e3} is chosen, we can extend $\Phi$ to $A^{\bot_s}$ by simply setting its value at $x\in A^{\bot_s}$ to be $\Phi([x])$. Let $\Gamma_\pm x$ be the $H_\pm$-part of $\Phi(x)$. Then we have the following abstract Green's formula
\begin{equation}[x,y]=i(\Gamma_+x,\Gamma_+y)_{H_+}-i(\Gamma_-x,\Gamma_-y)_{H_-},\quad \forall\ x,y\in A^{\bot_s}.\label{gre}\end{equation}
We shall call $(H_\pm, \Gamma_\pm)$ a \emph{boundary quadruple} for $A^{\bot_s}$. This generalizes the notion of boundary triplet for symmetric operators.

\section{Symmetric operators revisited}\label{sec3}
To motivate our approach to contractions and also for the convenience of comparison, in this section we sketch the basic formalism set up in \cite{wang2024complex} to deal with symmetric operators. The reader is referred to \cite{wang2024complex} for a detailed account.

\subsection{Symmetric operator and its Weyl curve}\label{sec3.1}
Let $H$ be a separable infinite-dimensional Hilbert space and $T$ a densely defined symmetric operator in $H$, i.e., $D(T)$ is dense in $H$ and $(Tx,y)=(x,Ty)$ for all $x,y\in D(T)$. Let $T^*$ be the conjugate of $T$. Basically $D(T)\subseteq D(T^*)$ and if further $D(T)=D(T^*)$, $T$ is called self-adjoint. We always assume that $T$ is closed, i.e., the graph of $T$ is a closed subspace of $\mathbb{H}:=H\oplus_\bot H$. With the graph inner product
\[(x,y)_{T}:=(x,y)+(T^*x,T^*y),\quad \forall\ x,\ y \in D(T^*), \]
$D(T^*)$ is a Hilbert space and $D(T)$ is a closed subspace. An extension of $T$ is an operator $\tilde{T}$ in $H$ such that $D(\tilde{T})\supseteq D(T)$ is a closed subspace of $D(T^*)$ and $\tilde{T}=T^*|_{D(\tilde{T})}$. Therefore, all extensions of $T$ are parameterized by points in the Grassmannian of closed subspaces of the quotient Hilbert space $\mathcal{B}_T:=D(T^*)/D(T)$.
$\mathcal{B}_T$ is a strong symplectic Hilbert space whose strong symplectic structure $[\cdot, \cdot]_T$ is defined as
\[[[x],[y]]_T=(T^*x,y)_H-(x,T^*y)_H.\]
Let $(n_+,n_-)$ be the signature of this strong symplectic structure. These are precisely the deficiency indices of $T$. We only consider the case that both $n_\pm$ are nonzero.

To facilitate later developments, we reformulate the above argument in another unfamiliar but equivalent form. Equip $\mathbb{H}$ with the non-standard strong symplectic structure $[\cdot, \cdot]_{new}$ in Example \ref{e5} and let $A_T\subseteq \mathbb{H}$ be the graph of $T$. Then that $T$ is symmetric means precisely that $A_T$ is isotropic. In this way, $A_T^{\bot_s}$ is precisely the graph of $T^*$ and $\mathcal{B}_T$ can be identified with the quotient strong symplectic Hilbert space $A_T^{\bot_s}/A_T$. An extension of $T$ now can be interpreted as an intermediate closed subspace between $A_T$ and $A_T^{\bot_s}$. Of particular interest in this paper are those extensions parameterized by maximal positive-definite subspaces of $\mathcal{B}_T$.

$\mathbb{H}$ has its non-standard polarization, i.e., $\mathbb{H}:=H\oplus_\bot H$. W.r.t. this, for each $\lambda\in \mathbb{C}$, define
\[W_\lambda:=\{(x,\lambda x)\in \mathbb{H}|x\in H\}.\]
It is easy to find that $W_\lambda\in W_\pm(\mathbb{H})$ for $\lambda\in \mathbb{C}_\pm $. Since $\mathbb{C}$ is of complex dimension 1, we call $W_\lambda$ with $\lambda\in \mathbb{C}$ the universal Weyl curve associated to the non-standard polarization of $\mathbb{H}$. For each $\lambda\in \mathbb{C}$, now define
\[N_\lambda:=W_\lambda\cap A_T^{\bot_s}=\{(x,\lambda x)\in \mathbb{H}|x\in D(T^*),\ T^*x=\lambda x\}.\]
It can be proved that for $\lambda\in \mathbb{C}_\pm$,
$\textup{dim}N_\lambda=n_\pm$. Let $W_T(\lambda)$ be the image of $N_\lambda$ in $\mathcal{B}_T$ under the quotient map $A_T^{\bot_s}\rightarrow \mathcal{B}_T$. Then $W_T(\lambda)$ is a maximal completely positive-definite (resp. negative-definite) subspace of $\mathcal{B}_T$ if $\lambda\in \mathbb{C}_+$ (resp. $\lambda\in \mathbb{C}_-$). It can also be proved that $(W_T(\lambda))^{\bot_s}=W_T(\bar{\lambda})$ for each $\lambda\in \mathbb{C}_+$. As a consequence, we have a \emph{holomorphic} curve
\[W_T(\lambda):\mathbb{C}_+\cup \mathbb{C}_-\rightarrow W_+(\mathcal{B}_T)\cup W_-(\mathcal{B}_T), \quad \lambda\mapsto W_T(\lambda).\]
 This $W_T(\lambda)$ is called the (two-branched) \emph{Weyl curve} of $T$, containing all unitarily invariant information of $T$. We also call the branch on $\mathbb{C}_+$ the (single-branched) Weyl curve because it totally determines the two-branched version.

 Though $\mathbb{H}$ is naturally polarized, the polarization is lost on the quotient $\mathcal{B}_T=A_T^{\bot_s}/A_T$. If a boundary quadruple $(H_\pm, \Gamma_\pm)$ is given for $A_T^{\bot_s}$, then for $\lambda\in \mathbb{C}_+$,
\[(\Gamma_+, \Gamma_-)W_T(\lambda)=\{(x,B(\lambda)x)\in H_+\oplus_\bot H_-|x\in H_+\}\]
where $B(\lambda)$ is a $\mathbb{B}(H_+,H_-)$-valued analytic function on $\mathbb{C}_+$ such that $\|B(\lambda)\|<1$. This so-called contractive Weyl function is thus an operator-valued Schur function on $\mathbb{C}_+$. Due to the fact $(W_T(\lambda))^{\bot_s}=W_T(\bar{\lambda})$,
\[(\Gamma_+, \Gamma_-)W_T(\lambda)=\{(B(\bar{\lambda})^*x,x)\in H_+\oplus_\bot H_-|x\in H_-\}\]
for $\lambda\in \mathbb{C}_-$.

Note that $\dim \textup{ker}(T^*-\lambda)=n_\pm$ for any $\lambda\in \mathbb{C}_\pm$. If $\textup{span}\{\cup_{\lambda\in \mathbb{C}\backslash \mathbb{R}}\textup{ker}(T^*-\lambda)\}$ is dense in $H$, $T$ is said to be simple. Simplicity of $T$ means precisely that $T$ cannot be written as the orthogonal direct sum of a nontrivial self-adjoint operator and another symmetric operator \cite[\S~3.4]{behrndt2020boundary}. A symmetric operator $T$ can always be decomposed into the direct sum of a self-adjoint part and a simple part,  and its Weyl curve $W_T(\lambda)$ only captures the information of its simple part. From now on, we only consider simple symmetric operators. We should also mention that the above formalism applies to simple symmetric operators that are not densely defined as well.
\subsection{A canonical functional model for symmetric operators}\label{sec3.2}
Though the use of reproducing kernel Hilbert spaces is well-known when people try to construct functional models for simple symmetric operators, before the work \cite{wang2024complex} only reproducing kernel Hilbert spaces of vector-valued functions were applied. In our opinion the general (but possibly less popular) vector bundle version is much more natural for dealing with symmetric operators. Reproducing kernels on vector bundles were proposed in \cite{bertram1998reproducing} and the paper \cite{koranyi2011classification} also contains a concise introduction to the topic when the underlying bundle is holomorphic. For our purpose, the focus is mainly on holomorphic vector bundles over an open subset $\Omega\subset\mathbb{C}$ and the relevant Hilbert spaces then consist of holomorphic sections of the underlying vector bundles. For basics on holomorphic vector bundles, we refer the reader to \cite{griffiths2014principles}.

Assume that $E\rightarrow M$ is a complex vector bundle\footnote{We allow the rank of $E$ to be infinite.} over a topological space $M$ and $\mathcal{H}$ a Hilbert space of continuous sections of $E$. Let $E^\dag$ be the conjugate-linear dual\footnote{We make the natural identification $(E^\dag)^\dag=E$.} of $E$. The pairing between $\varphi\in E_x^\dag$ and $\psi\in E_x$ for $x\in M$ will be denoted by $\varphi(\psi)=((\varphi, \psi))$. Suppose the evaluation map \[ev_x: \mathcal{H}\rightarrow E_x\subseteq E, \ s\mapsto s(x)\]
is continuous for any $x\in M$. By Riesz representation theorem, we have a $\mathbb{C}$-linear map $ev_x^\dag: E_x^\dag\rightarrow \mathcal{H}$ defined by (the so-called reproducing property)
 \[(ev_x^\dag(\varphi),v)_\mathcal{H}=((\varphi, ev_x(v)))=((\varphi, v(x)))\]
 for any $\varphi\in E_x^\dag$ and $v\in \mathcal{H}$.
 We set $K(x,y)=ev_xev_y^\dag$, which is a linear map from $E^\dag_y$ to $E_x$ and called the reproducing kernel of $\mathcal{H}$ on $E$. $K(x,y)$ is positive in the following sense: For any points $x_j\in M$, $j=1,\cdots, p$ and $\varphi_j\in E_{x_j}^\dag$, we have \[\sum_{i,j}((\varphi_i,K(x_i,x_j)\varphi_j))\geq 0,\] which is trivially
 \[(\sum_iev^\dag_{x_i}\varphi_i,\sum_jev^\dag_{x_j}\varphi_j)_\mathcal{H}\geq 0.\]
 Conversely, any continuous and positive kernel $K$ on $E$ is always the reproducing kernel of a Hilbert space $\mathfrak{H}(K)$ of continuous sections of $E$. $\mathcal{H}$ can be constructed as follows. Fix $y\in M$ and $\varphi\in E^\dag_y$, and let $x$ run over $M$. Then $K(\cdot, y)\varphi$ is a continuous section of $E$. Define
 $$\breve{\mathfrak{H}}(K):=\textup{span}\{K(\cdot, y)\varphi|y\in M, \varphi\in E^\dag_y\}.$$
  On $\breve{\mathfrak{H}}(K)$, we can give a (positive-definite) inner product defined by
 \[(K(\cdot, y_1)\varphi_1,K(\cdot, y_2)\varphi_2)_K=((K(y_2,y_1)\varphi_1,\varphi_2))\]
 for $\varphi_i\in E^\dag_{y_i}$, $i=1,2$. Completing $\breve{\mathfrak{H}}(K)$ w.r.t. this inner product then gives rise to the required $\mathfrak{H}(K)$. We shall say $\mathfrak{H}(K)$ is generated by the kernel $K$.

 Now if $E$ is holomorphic over a complex manifold $M$ and $\mathcal{H}$ a reproducing kernel Hilbert space of holomorphic sections of $E$, then the kernel $K(x,y)$ is additionally
 \begin{itemize}
   \item sesqui-analytic, i.e., holomorphic in $x$, anti-holomorphic in $y$,
     \item locally uniformly bounded in the sense that $K(x,x)$ is uniformly bounded over any compact subset of $M$.
 \end{itemize}
 Conversely, if a positive kernel $K(x,y)$ on $E$ satisfying these properties is given, the sections of $E$ in the corresponding reproducing kernel Hilbert space $\mathfrak{H}(K)$ are holomorphic.

In this formalism, \cite{wang2024complex} constructed a canonical functional model for simple symmetric operators in a rather direct way and didn't use the contractive Weyl function explicitly. The construction goes as follows.

Given a simple symmetric operator $T$ in $H$ with deficiency indices $(n_+,n_-)$, on $\mathbb{C}_+\cup \mathbb{C}_-$ there is a natural holomorphic Hermitian vector bundle $E$: At $\lambda\in \mathbb{C}_+\cup \mathbb{C}_-$, the fiber $E_\lambda$ is just $\textup{ker}(T^*-\lambda)$ and the Hermitian structure on $E_\lambda$ is obtained by restricting the inner product on $H$. That $E$ is holomorphic is a consequence of the holomorphicity of $W_T(\lambda)$. Exchanging the fibers of $E$ at $\lambda$ and $\bar{\lambda}$, we obtain an anti-holomorphic vector bundle $F^\dag$. Let $F$ be the conjugate linear dual of $F^\dag$. Then $F$ is again a holomorphic vector bundle equipped with the induced Hermitian structure.

Let $\iota_\lambda: F_\lambda^\dag \rightarrow H$ be the natural inclusion and $\iota_\lambda^\dag$ the conjugate of $\iota_\lambda$ defined by
\[((\iota_\lambda^\dag x, \omega))=(x,\iota_\lambda \omega)_H=(x,\omega)_H,\quad \forall\ \omega\in F^\dag_\lambda.\]
Then we get a reproducing kernel $K(\lambda,\mu):=\iota_\lambda^\dag \iota_\mu: F^\dag_\mu\rightarrow F_\lambda$. The reproducing kernel Hilbert space generated by $K(\lambda,\mu)$ is just our model space $\mathfrak{H}$. Elements in $\mathfrak{H}$ can be described in a more direct way: Given $x\in H$, we can associate a holomorphic section $\hat{x}$ of $F$. At $\lambda$ we define $\hat{x}(\lambda)$ via
\[((\hat{x}(\lambda), \omega))=(x, \iota_\lambda \omega)_H,\quad \forall\ \omega\in F^\dag_\lambda.\]
In this way, obviously $\hat{x}(\cdot)=\iota_\cdot^\dag x$. Due to simplicity of $T$, this map $x\mapsto \hat{x}$ is injective and actually a unitary map from $H$ to $\mathfrak{H}$. Let $\mathcal{X}$ be the multiplication by the independent variable $\lambda$ on holomorphic sections of $F$ and define
\[D(\mathfrak{T}):=\{s\in \mathfrak{H}|\mathcal{X}s\in \mathfrak{H}\}.\]
Denote the restriction of $\mathcal{X}$ on $D(\mathfrak{T})$ by $\mathfrak{T}$. Then $\mathfrak{T}$ is the model operator of $T$, i.e., $T$ and $\mathfrak{T}$ are unitarily equivalent via the "Fourier transform" $x\mapsto \hat{x}$.

Alternatively, the model can also be constructed from the Weyl curve $W_T(\lambda)$. More generally, let $H$ be a strong symplectic Hilbert space with signature $(n_+,n_-)$ and $W_+(H)$ the space of maximal completely positive-definite subspaces. A holomorphic map $N(\lambda):\mathbb{C}_+\rightarrow W_+(H)$ is called a Nevanlinna curve. We can construct a simple (not necessarily densely defined) symmetric operator $\mathfrak{T}$ from any given Nevanlinna curve $N(\lambda)$ as follows.

We extend $N(\lambda)$ to $\mathbb{C}_+\cup \mathbb{C}_-$ by setting $N(\lambda)=(N(\bar{\lambda}))^{\bot_s}\in W_-(H)$ for $\lambda\in \mathbb{C}_-$. $W_+(H)\cup W_-(H)$ as a manifold has a tautological holomorphic Hermitian vector bundle, i.e., at $p\in W_+(H)\cup W_-(H)$, the fiber is precisely the subspace of $H$ parameterized by $p$ and the Hermitian structure is just the one obtained by restricting the strong symplectic structure. This vector bundle can be pulled back via $N(\lambda)$ and results in a holomorphic Hermitian vector bundle $E$ over $\mathbb{C}_+\cup \mathbb{C}_-$. Exchanging the fibers of $E$ at $\lambda$ and $\bar{\lambda}$ for each $\lambda\in \mathbb{C}_+$ produces a vector bundle $F^\dag$, which is anti-holomorphic. Let $F$ be the conjugate linear dual of $F^\dag$.

Note that at each $\lambda\in \mathbb{C}_+\cup \mathbb{C}_-$ there is the decomposition $H=F^\dag_\lambda\oplus F^\dag_{\bar{\lambda}}$. Denote the projection along $F^\dag_{\bar{\lambda}}$ onto $F^\dag_\lambda$ by $\mathcal{P}_\lambda$. For $\lambda, \mu\in \mathbb{C}_+\cup \mathbb{C}_-$, we can construct a map $\mathcal{Q}(\lambda, \mu)$ from $F^\dag_\mu$ to $F_\lambda$: If $\omega\in F^\dag_\mu$,
then $\mathcal{P}_\lambda\omega\in F^\dag_\lambda$ and $(\mathcal{P}_\lambda\omega,\cdot)_\lambda$ (the Hermitian structure on the fiber is used) is a conjugate-linear functional on $F^\dag_\lambda$. Consequently, $(\mathcal{P}_\lambda\omega,\cdot)_\lambda\in F_\lambda$ and we simply set $\mathcal{Q}(\lambda, \mu)\omega=(\mathcal{P}_\lambda\omega,\cdot)_\lambda$. Now for $\lambda\neq \bar{\mu}$ we define
\[K(\lambda,\mu)=i\frac{\mathcal{Q}(\lambda, \mu)}{\lambda-\bar{\mu}}.\]
As for $\lambda=\bar{\mu}$, $\mathcal{Q}(\bar{\mu},\mu)$ is certainly zero. It turns out that $\mathcal{Q}(\lambda, \mu)$ as an analytic function in $\lambda$ has a zero at $\bar{\mu}$, and we can simply set $K(\bar{\mu},\mu)$ to be the limit of the above formula as $\lambda$ approaches $\bar{\mu}$.

In terms of a chosen polarization of $H$, it can be checked that the above $K(\lambda,\mu)$ is a reproducing kernel on the holomorphic vector bundle $F$. Let $\mathfrak{H}$ be the reproducing kernel Hilbert space generated by $K(\lambda, \mu)$. Then multiplication by the independent variable is a simple symmetric operator $\mathfrak{T}$ in $\mathfrak{H}$. The Weyl curve of $\mathfrak{T}$ is congruent to $N(\lambda)$, i.e., they differ from each other by a symplectic isomorphism between the underlying strong symplectic Hilbert spaces. In particular, if $N(\lambda)$ is the Weyl curve of a given simple symmetric operator $T$, then $\mathfrak{T}$ is unitarily equivalent to $T$.
\section{A functional model for contractions}\label{sec4}
In this section, we show how the methodology of \cite{wang2024complex} can be adapted to give a new functional model for contractions. The construction is direct in the sense that the model space $\mathfrak{H}$ can be derived from natural geometric data associated to the contraction under consideration, without mentioning the characteristic function $\Theta_T(\lambda)$ at all. In either the Sz.-Nagy-Foias model or the de Branges-Rovnyak model, $\Theta_T(\lambda)$ must be obtained first and then is used partially to construct $\mathfrak{H}$.
\subsection{Some preliminaries}\label{sec4.1}
On an infinite-dimensional separable Hilbert space $H$, $T\in \mathbb{B}(H)$ is called a contraction in case $\|T\|\leq 1$. $T$ is called completely non-unitary (c.n.u. for short) if $T$ has no nontrivial invariant subspace $\mathcal{K}$ such that $T|_\mathcal{K}$ is unitary. Any contraction can be decomposed into the orthogonal direct sum of a unitary part and a c.n.u. part, and in this section we only consider c.n.u. contractions. A basic reference on contractions is \cite{nagy2010harmonic}.

Let $T$ be a c.n.u. contraction on $H$, and
\[\mathbb{K}:=\textup{ker}(Id-T^*T),\quad \mathbb{K_*}:=\textup{ker}(Id-TT^*).\]
Due to the basic identity $T(Id-T^*T)=(Id-TT^*)T$ and its conjugate counterpart, we see that
\[T\mathbb{K}\subseteq \mathbb{K}_*,\quad T^*\mathbb{K}_*\subseteq \mathbb{K}.\]
 In particular, $T|_{\mathbb{K}}$ is a unitary map from $\mathbb{K}$ onto $\mathbb{K}_\ast$. We call $T|_{\mathbb{K}}$ the \emph{unitary map part} of $T$. It is also easy to see $T$ maps $\mathbb{K}^\bot$ into $\mathbb{K}_\ast^\bot$ and $T^*$ maps $\mathbb{K}_\ast^\bot$ into $\mathbb{K}^\bot$ likewise. We shall denote $T|_{\mathbb{K}^\bot}$ by $t$ and call it the \emph{non-unitary map part} of $T$. The two numbers $n_+=\dim \mathbb{K}^\bot$ and $n_-=\dim \mathbb{K}_*^\bot$ are called the deficiency indices of $T$.\footnote{Some authors call them the defect numbers, e.g., \cite{nagy2010harmonic}.}

In both the Sz.-Nagy-Foias model and the de Branges-Rovnyak model, the so-called characteristic function $\Theta_T(\lambda)$ of $T$ plays a central role in constructing a functional model for $T$. Let $\mathfrak{D}=(Id-T^*T)^{1/2}$ and $\mathfrak{D}_\ast=(Id-TT^*)^{1/2}$. Then
\[\Theta_T(\lambda):=(-T+\lambda \mathfrak{D}_*(Id-\lambda T^*)^{-1}\mathfrak{D})|_{\mathbb{K}^\bot}.\]
$\Theta_T(\lambda)$ is actually a $\mathbb{B}(\mathbb{K}^\bot, \mathbb{K}_\ast ^\bot)$-valued analytic function on the unit disc $\mathbb{D}$. In particular, for each $\lambda\in \mathbb{D}$, $\Theta_T(\lambda)$ is a pure contraction. For the motivation of this definition, see \cite{nagy2010harmonic, douglas1974canonical}. For more details on its role in the Sz.-Nagy-Foias model and the de Branges-Rovnyak model, see \cite[Chap.~VI]{nagy2010harmonic} and \cite{ball2014branges}.

We shall take an alternative viewpoint towards the above material. Let $\mathbb{H}=H\oplus_\bot H$ be the standard strong symplectic Hilbert space in Example \ref{e3} with $H_+=H_-=H$. Then obviously, the unitary map part of $T$ defines an isotropic subspace $A_T$ of $\mathbb{H}$:
\[A_T:=\{(x,Tx)\in \mathbb{H}|x\in \mathbb{K}\}=\{(T^*x,x)\in \mathbb{H}|x\in \mathbb{K}_*\}.\]
\begin{lemma}\label{l1}In the strong symplectic Hilbert space $\mathbb{H}$,
\[A_T^{\bot_s}=\{(y_1, y_2)\in \mathbb{H}|y_1-T^*y_2\in \mathbb{K}^\bot\}=\{(y_1, y_2)\in \mathbb{H}|Ty_1-y_2\in \mathbb{K}_*^\bot\}.\]
\end{lemma}
\begin{proof}Obvious.
\end{proof}
$\mathbb{\mathbb{K}}^\bot$ and $\mathbb{K}_*^\bot$ are naturally embedded in $\mathbb{H}$ via $x\mapsto (x,0)\in \mathbb{H}$ and $x\mapsto (0, x)\in \mathbb{H}$ respectively. Denote the corresponding images by $Q$ and $Q_*$.
\begin{lemma}\label{l3}In the strong symplectic Hilbert space $\mathbb{H}$,
\[A_T^{\bot_s}=A_T\oplus_\bot Q \oplus_\bot Q_*.\]
\end{lemma}
\begin{proof}That the direct sum is orthogonal is obvious. It is clear that $A_T\oplus_\bot Q \oplus_\bot Q_*\subseteq A_T^{\bot_s}$. Let $(y_1,y_2)\in A_T^{\bot_s}$ be orthogonal to $A_T$. Then for any $x\in \mathbb{K}$, $(y_1,x)_H-(y_2,Tx)_H=0$ and $(y_1,x)_H+(y_2,Tx)_H=0$. Thus $(y_1,x)_H=(y_2,Tx)_H=0$ for any $x\in \mathbb{K}$, i.e., $y_1\in \mathbb{K}^\bot$ and $y_2\in \mathbb{K}_*^\bot$. The proof is finished.
\end{proof}
For $a\in A_T^{\bot_s}$, if $(a_+,0)\in Q$ (resp. $(0,a_-)\in Q_*$) is the $Q$-part (resp. $Q_*$-part) of $a$ w.r.t. the above decomposition, we set $\Gamma_\pm a=a_\pm$.
\begin{lemma}$(\mathbb{K}^\bot, \mathbb{K}_*^\bot, \Gamma_\pm)$ is a boundary quadruple for $A_T^{\bot_s}$.
\end{lemma}
\begin{proof}This can be checked directly.
\end{proof}
We shall call this boundary quadruple the canonical one associated to $T$. Clearly, the graph $Gr_T$ of $T$ is an intermediate subspace between $A_T$ and $A_T^{\bot_s}$. $Gr_T$ can be singled out from $A_T^{\bot_s}$ by requiring $\Gamma_-a=t\Gamma_+a$ for $a\in A_T^{\bot_s}$, where $t=T|_{\mathbb{K}^\bot}$.

With the standard polarization of $\mathbb{H}$, we can again define the associated universal Weyl curve
\[W_\lambda:=\{(x,\lambda x)\in \mathbb{H}|x\in H\}.\]
However, since the strong symplectic structure is different from the one used in \S~\ref{sec3.1}, this time we find that $W_\lambda\in W_+(\mathbb{H})$ for $\lambda\in \mathbb{D}$ while $W_\lambda\in W_-(\mathbb{H})$ for $\lambda\in \mathbb{C}\backslash \overline{\mathbb{D}}$. We are interested in $N_\lambda:=W_\lambda \cap A_T^{\bot_s}$.
\begin{proposition}\label{p1}For $\lambda\in \mathbb{D}$,
\[N_\lambda=\{((Id-\lambda T^*)^{-1}f, \lambda (Id-\lambda T^*)^{-1}f)\in \mathbb{H}|f\in \mathbb{K}^\bot\},\]
and for $\lambda\in \mathbb{C}\backslash \overline{\mathbb{D}}$,
\[N_\lambda=\{((\lambda-T)^{-1}f, \lambda (\lambda- T)^{-1}f)\in \mathbb{H}|f\in \mathbb{K}_*^\bot\}.\]
\end{proposition}
\begin{proof}For $\lambda\in \mathbb{D}$ and $x\in H$, by Lemma \ref{l1} $(x,\lambda x)\in A_T^{\bot_s}$ if and only if
\[x-\lambda T^* x\in \mathbb{K}^{\bot}.\]
Thus the first claim follows. The second can be checked similarly.
\end{proof}
For later convenience, we shall denote $\mathbb{D}$ by $\mathbb{D}_+$ and identify $(\mathbb{C}\backslash \overline{\mathbb{D}})\cup \{\infty\}$ with another copy of $\mathbb{D}$ (denoted by $\mathbb{D}_-$) via $\lambda \mapsto 1/\lambda$. In this way, for $\lambda\in \mathbb{D}_-$, $N_\lambda$ has the following form:
\[N_\lambda=\{(\lambda(Id-\lambda T)^{-1}f,  (Id- \lambda T)^{-1}f)\in \mathbb{H}|f\in \mathbb{K}_*^\bot\}.\]
In particular, this even makes sense for $\lambda=0$. Also to avoid ambiguities, $0\in \mathbb{D}_\pm$ will be denoted by $0_\pm$ respectively from now on.

We are now in a position to explain the relation between contractions and symmetric operators. For a simple symmetric operator $T$, the graph $A_T$ is an isotropic subspace of $(\mathbb{H}, [\cdot, \cdot]_{new})$. The Cayley transform in Example \ref{e8} turns $A_T$ into an isotropic subspace $A$ in $(\mathbb{H}, [\cdot, \cdot]_{st})$, and $A_T^{\bot_s}$ into $A^{\bot_s}$. In terms of the standard polarization of $(\mathbb{H}, [\cdot, \cdot]_{st})$, $A$ is the graph of $C:=(T-i)(T+i)^{-1}$ which is a unitary map from $\textup{Ran}(T+i)$ to $\textup{Ran}(T-i)$. The Cayley transform also turns the universal Weyl curve $W_\lambda^1$ associated to $(\mathbb{H}, [\cdot, \cdot]_{new})$ into the universal Weyl curve $W_\lambda^2$ associated to $(\mathbb{H}, [\cdot, \cdot]_{st})$. Correspondingly, $W_\lambda^1\cap A_T^{\bot_s}$ is turned into $W_\mu^2\cap A^{\bot_s}$ where $\mu=\frac{\lambda-i}{\lambda+i}$.

Conversely, given a c.n.u. contraction $C$ on $H$, its unitary map part gives rise to an isotropic subspace $A$ in $(\mathbb{H}, [\cdot, \cdot]_{st})$ while the \emph{inverse} Cayley transform turns $A$ into the graph of a simple (not necessarily densely defined) symmetric operator $T$. The statement that $C$ is c.n.u. means precisely that $T$ is simple. The non-unitary map part of $C$ then corresponds to a dissipative extension of $T$ which is parameterized by a maximal positive-definite subspace of $A^{\bot_s}/A$.

In fact, Cayley transform was a basic tool in von Neumann's original approach to symmetric operators, reducing the problem of self-adjoint extensions of $T$ to the problem of unitary extensions of $(T-i)(T+i)^{-1}$. For an up-to-date account of the use of Cayley transform, see \cite[Chap.~1]{behrndt2020boundary}.

Bearing this correspondence via Cayley transform in mind, let us come back to the theory of c.n.u. contractions.
\begin{theorem}For $\lambda\in \mathbb{D}_+\cup \mathbb{D}_-$, let $N_\lambda$ be as in Prop.~\ref{p1} and the remark after it. Denote the image of $N_\lambda$ in $\mathcal{B}_T:=A_T^{\bot_s}/A_T$ under the quotient map by $W_T(\lambda)$. This defines a holomorphic map \[W_T:\mathbb{D}_+\cup \mathbb{D}_- \rightarrow W_+(\mathcal{B}_T)\cup W_-(\mathcal{B}_T)\]
such that $W_T(\bar{\lambda})=(W_T(\lambda))^{\bot_s}$, where $\lambda\in \mathbb{D}_+$ and $\bar{\lambda}$ is its conjugate but interpreted as an element in $\mathbb{D}_-$.
\end{theorem}
\begin{proof}The proof can be carried out directly, but we can alternatively apply the inverse Cayley transform to the unitary map part of $T$ and obtain a simple symmetric operator. Then the theory developed in \cite{wang2024complex} immediately implies the claim.
\end{proof}
Following \cite{wang2024complex}, we call $W_T(\lambda)$ or its branch over $\mathbb{D}_+$ the Weyl curve of $T$. It should be emphasized that $W_T(\lambda)$ only contains unitarily invariant information of the unitary map part of $T$ while the additional $t=T|_{\mathbb{K}^\bot}$ parameterizes a specific contractive extension of $T|_\mathbb{K}$ in terms of the canonical boundary quadruple
$(\mathbb{K}^\bot, \mathbb{K}_*^\bot, \Gamma_\pm)$. In this sense, we shall also call $W_T(\lambda)$ the Weyl curve of $A_T$.

In general, given a boundary quadruple $(H_\pm, \Gamma_\pm)$ for $A_T^{\bot_s}$, then for $\lambda\in \mathbb{D}_+$
\[(\Gamma_+,\Gamma_-)W_T(\lambda)=\{(x,B(\lambda)x)\in H_+\oplus_\bot H_-|x\in H_+\}\]
and for $\lambda\in \mathbb{D}_-$,
\[(\Gamma_+,\Gamma_-)W_T(\lambda)=\{(B(\bar{\lambda})^*x,x)\in H_+\oplus_\bot H_-|x\in H_-\}\]
where $B(\lambda)$ is an analytic $\mathbb{B}(H_+,H_-)$-valued Schur function on $\mathbb{D}_+$ such that $\|B(\lambda)\|<1$ pointwise. $B(\lambda)$ will be called the \emph{contractive Weyl function} of $T$ w.r.t. the given boundary quadruple. Another way to interpret $B(\lambda)$ is as follows: Given $x\in H_+$ and $\lambda\in \mathbb{D}_+$, there is a unique $\gamma_+(\lambda)x\in N_\lambda$ such that $\Gamma_+\gamma_+(\lambda)x=x$. Then $B(\lambda)x=\Gamma_-\gamma_+(\lambda)x$. Similarly, for $x\in H_-$ and $\lambda\in \mathbb{D}_-$, there is a unique $\gamma_-(\lambda)x\in N_\lambda$ such that $\Gamma_-\gamma_-(\lambda)x=x$. Then $B(\bar{\lambda})^*x=\Gamma_+\gamma_-(\lambda)x$. The two maps $\gamma_\pm(\lambda):\mathbb{D}_\pm\rightarrow \mathbb{B}(H_\pm, \mathbb{H})$ are holomorphic and will be called the $\gamma$-\emph{fields} associated to the boundary quadruple $(H_\pm, \Gamma_\pm)$.

\begin{example}\label{e7}Let us consider the contractive Weyl function associated to the canonical boundary quadruple $(\mathbb{K}^\bot, \mathbb{K}_*^\bot, \Gamma_\pm)$. For $\lambda\in \mathbb{D}_+$, let $a=((Id-\lambda T^*)^{-1}f, \lambda (Id-\lambda T^*)^{-1}f)\in N_\lambda$ for some $f\in \mathbb{K}^\bot$ and $(x_0, Tx_0)$ be the $A_T$-part of $a$ w.r.t. the decomposition in Lemma \ref{l3}. By definition, we have
\begin{equation}(Id-\lambda T^*)^{-1}f=x_0+a_+,\quad \lambda (Id-\lambda T^*)^{-1}f=Tx_0+a_-\label{eq1}\end{equation}
where $a_-=B(\lambda)a_+$. Denote the restriction of $\mathfrak{D}$ on $\mathbb{K}^\bot$ by $\mathfrak{D}_t$. We can assume that $\mathfrak{D}_t$ has a bounded inverse, otherwise scaling $t=T|_{\mathbb{K}^\bot}$ slightly will do the job. We should have $\mathfrak{D}(Id-\lambda T^*)^{-1}f=\mathfrak{D}_ta_+$. Then
\[a_+=\mathfrak{D}_t^{-1}\mathfrak{D}(Id-\lambda T^*)^{-1}f.\]
Solving this equation we get $f=R(\lambda)a_+$ where $R(\lambda)$ is the inverse of $\mathfrak{D}_t^{-1}\mathfrak{D}(Id-\lambda T^*)^{-1}$. In terms of these, one can find
\[B(\lambda)a_+=ta_+-(T-\lambda)(Id-\lambda T^*)^{-1}R(\lambda)a_+.\]
It is clear that for this boundary quadruple $B(0_+)=0$.
\end{example}
\begin{example}\label{e6}If $t=T|_{\mathbb{K}^\bot}\in \mathbb{B}(\mathbb{K}^\bot,\mathbb{K}_\ast^\bot)$ satisfies $\|t\|<1$, then $t$ itself provides a polarization of $\mathcal{B}_T$. We can use it to produce another boundary quadruple $(\mathbb{K}^\bot, \mathbb{K}_*^\bot, \Gamma_\pm')$. Use the pseudo-unitary transform
\[\left(
    \begin{array}{cc}
      (Id-t^*t)^{-1/2} & -t^*(Id-tt^*)^{-1/2} \\
      t(Id-t^*t)^{-1/2} & -(Id-tt^*)^{-1/2} \\
    \end{array}
  \right)\]
on $\mathbb{K}^\bot\oplus_\bot\mathbb{K}_*^\bot$ to define
\[\Gamma_+'=(Id-t^*t)^{-1/2}\Gamma_+-t^*(Id-tt^*)^{-1/2}\Gamma_-,\quad \Gamma_-'= t(Id-t^*t)^{-1/2}\Gamma_+ -(Id-tt^*)^{-1/2}\Gamma_-,\]
where $(\mathbb{K}^\bot, \mathbb{K}_*^\bot, \Gamma_\pm)$ is the canonical boundary quadruple. Note that the graph $Gr_T$ of $T$ lying in $A_T^{\bot_s}$ is singled out by the requirement $\Gamma_-'a=0$.
\end{example}
The following proposition shows how the Sz.-Nagy-Foias characteristic function makes its appearance in our formalism.
\begin{proposition}\label{p3}In Example \ref{e6}, w.r.t. the new boundary quadruple $(\mathbb{K}^\bot, \mathbb{K}_*^\bot, \Gamma_\pm')$, the corresponding contraction Weyl function is
\[B'(\lambda)=-\Theta_T(\lambda)=(T-\lambda\mathfrak{D}_*(Id-\lambda T^*)^{-1}\mathfrak{D})|_{\mathbb{K}^\bot}.\]
\end{proposition}
\begin{proof}We still use the notation in Example \ref{e7}. For $a=((Id-\lambda T^*)^{-1}f, \lambda (Id-\lambda T^*)^{-1}f)$, by definiton
\[-\Theta_T(\lambda)\Gamma_+'a=t((Id-t^*t)^{-1/2}a_+-t^*(Id-tt^*)^{-1/2}a_-)-\lambda\mathfrak{D}_*(Id-\lambda T^*)^{-1}(a_+-t^*a_-)\]
and
\[\Gamma'_-a= t(Id-t^*t)^{-1/2}a_+ -(Id-tt^*)^{-1/2}a_-.\]
Thus
\begin{eqnarray*}-\Theta_T(\lambda)\Gamma_+'a-\Gamma'_-a&=&(Id-tt^*)(Id-tt^*)^{-1/2}a_--\lambda\mathfrak{D}_*(Id-\lambda T^*)^{-1}(a_+-t^*a_-)\\
&=&(Id-tt^*)^{1/2}a_--\lambda\mathfrak{D}_*(Id-\lambda T^*)^{-1}(a_+-t^*a_-).
\end{eqnarray*}
Due to Eq.~(\ref{eq1}), we have
\[\lambda T^*(Id-\lambda T^*)^{-1}f=T^*Tx_0+T^*a_-=x_0+t^*a_-\]
and thus $f=a_+-t^*a_-$. So we have
\[-\Theta_T(\lambda)\Gamma_+'a-\Gamma'_-a=(Id-tt^*)^{1/2}a_--\lambda\mathfrak{D}_*(Id-\lambda T^*)^{-1}f.\]
Again due to Eq.~(\ref{eq1}),
\[\lambda\mathfrak{D}_*(Id-\lambda T^*)^{-1}f=\mathfrak{D}_*a_-=(Id-tt^*)^{1/2}a_-\]
and we finally get $-\Theta_T(\lambda)\Gamma_+'a-\Gamma'_-a=0$. The claim is thus proved.
\end{proof}
\emph{Remark}. The Sz.-Nagy-Foias characteristic function $\Theta_T(\lambda)$ also makes sense for c.n.u. contractions $T$ such that $\|t\|=1$, but in this latter case it cannot appear as the contractive Weyl function simply because now $t$ won't provide a polarization for $\mathcal{B}_T$ due to Prop.~\ref{p4}.

\subsection{Characteristic vector bundles and the functional model}\label{sec4.2}

We are now prepared for our functional model for a c.n.u. contraction $T$.

Given $T$, on $\mathbb{D}_+\cup \mathbb{D}_-$ there is a natural holomorphic Hermitian vector bundle $E$. At $\lambda\in \mathbb{D}_+$, $E_\lambda=pr_1(N_\lambda)$ and at $\lambda\in \mathbb{D}_-$, $E_\lambda=pr_2(N_\lambda)$ where $pr_i$ is the orthogonal projection of $\mathbb{H}$ onto the $i$-th copy of $H$. The Hermitian structure is obtained by restricting the inner product on $H$ to $E_\lambda$. Exchanging the fiber at each $\lambda\in \mathbb{D}_+$ with the fiber at $\bar{\lambda}\in \mathbb{D}_-$, we have an anti-holomorphic vector bundle $F^\dag$. Let $F$ be the conjugate linear dual of $F^\dag$. Then $F$ is a holomorphic vector bundle with the induced Hermitian structure.

There is a natural reproducing kernel $K(\lambda, \mu)$ on $F$: For $\lambda\in \mathbb{D}_+\cup \mathbb{D}_-$, let $\iota_\lambda:F^\dag_\lambda\rightarrow H$ be the canonical inclusion. Then we have the conjugate $\iota_\lambda^\dag: H\rightarrow F_\lambda$ defined by
\[((\iota_\lambda^\dag x, \omega))=(x , \iota_\lambda \omega)_H=(x, \omega)_H\]
for each $\omega\in F^\dag_\lambda$. The kernel $K(\lambda, \mu)$ is defined to be $K(\lambda, \mu)=\iota_\lambda^\dag \iota_\mu$. The positivity of $K(\lambda, \mu)$ is clear, while that it is sesqui-analytic and locally uniformly bounded will follow easily when a holomorphic trivialization of $E$ is chosen. Let $\mathfrak{H}$ be the reproducing kernel Hilbert space generated by $K(\lambda, \mu)$, consisting of certain holomorphic sections of $F$.

Elements in $\mathfrak{H}$ can be constructed more explicitly. For any $x\in H$,  we can construct a holomorphic section $\hat{x}$ of $F$ as follows. If $\lambda\in \mathbb{D}_+\cup \mathbb{D}_-$, then $\hat{x}(\lambda)$ is defined by
\[((\hat{x}(\lambda), \omega))=(x,\omega)_H\]
for all $\omega\in F^\dag_\lambda$. Obviously $\hat{x}(\lambda)=\iota_\lambda^\dag x$. The map $x\mapsto \hat{x}$ is injective since $T$ is c.n.u.. The Hilbert space structure on $H$ can be transported onto the image $\widehat{H}$ and actually $\widehat{H}=\mathfrak{H}$. This can be justified by the uniqueness of the reproducing kernel.

This $\mathfrak{H}$ will be our model space for $T$. The advantage of our construction is that it is complex geometric in nature and doesn't depend heavily on the Euclidean properties of the unit disc $\mathbb{D}$. By contrast, in both the Sz.Nagy-Foias and the de Branges-Rovnyak models, the viewpoint of harmonic analysis is basic and thus preliminary knowledge of vector-valued Hardy spaces is necessary to construct the model space. We don't need it.

\begin{lemma}\label{l2}For $x\in \mathbb{K}$, $\widehat{Tx}(\lambda)=\lambda \hat{x}(\lambda)$ for each $\lambda\in \mathbb{D}_+$ and $\widehat{Tx}(\lambda)= \hat{x}(\lambda)/\lambda$ for each $\lambda\in \mathbb{D}_-$. In particular, $\hat{x}(0_-)=0$ in the latter case.
\end{lemma}
\begin{proof}If $\omega\in F_\lambda^\dag$ where $\lambda\in \mathbb{D}_+$, then by definition $(\bar{\lambda}\omega, \omega)\in N_{\bar{\lambda}}$ where $\bar{\lambda}\in \mathbb{D}_-$ and
\[(x,\bar{\lambda}\omega)_H-(Tx,\omega)_H=0.\]
Thus
\[((\widehat{Tx}(\lambda), \omega))=((\iota_\lambda^\dag (Tx), \omega))=(Tx,\omega)_H=\lambda(x,\omega)_H=((\lambda \iota_\lambda^\dag x, \omega))=((\lambda \hat{x}(\lambda), \omega)),\]
and hence $\widehat{Tx}(\lambda)=\lambda \hat{x}(\lambda)$. If $\omega\in F_\lambda^\dag$ for $\lambda\in \mathbb{D}_-$, then by definition $(\omega, \bar{\lambda}\omega)\in N_{\bar{\lambda}}$ where $\bar{\lambda}\in \mathbb{D}_+$ and
\[(x,\omega)_H-(Tx,\bar{\lambda}\omega)_H=0.\]
We can deduce that $\hat{x}(\lambda)=\lambda\widehat{Tx}(\lambda)$ and in particular, $\hat{x}(0_-)=0$.
\end{proof}
For a subspace $S\subseteq \mathbb{H}$, we use $\hat{S}$ to denote its "Fourier transform" in $\mathfrak{H}\oplus_\bot\mathfrak{H}$, i.e.,
\[\hat{S}=\{(\hat{x},\hat{y})\in \mathfrak{H}\oplus_\bot\mathfrak{H}|(x,y)\in S\}.\]
\begin{proposition}\label{p5}$(f,g)\in \widehat{A_T}\subseteq \mathfrak{H}\oplus_\bot\mathfrak{H}$ if and only if
\[g(\lambda)=\lambda f(\lambda)\]
for each $\lambda\in \mathbb{D}_+$
and
\[\lambda g(\lambda)=f(\lambda)\]
for each $\lambda\in \mathbb{D}_-$.
\end{proposition}
\begin{proof}The necessity part is just Lemma \ref{l2}. To prove the sufficiency part, let $(\hat{x},\hat{y})\in \mathfrak{H}\oplus_\bot\mathfrak{H}$ satisfy the listed formulae. Then for $0_-\in \mathbb{D}_-$, $\hat{x}(0_-)=0$. Note that in this case $F^\dag_{0_-}=\mathbb{K}^\bot$ and for any $\omega\in \mathbb{K}^\bot$
\[0=((\hat{x}(0_-), \omega))=((\iota_{0_-}^\dag x, \omega))=(x,\omega)_H.\]
Thus $x\in (\mathbb{K}^\bot)^{\bot}=\mathbb{K}$. Then Lemma \ref{l2} implies for $\lambda\in \mathbb{D}_+$
\[\widehat{Tx}(\lambda)=\lambda \hat{x}(\lambda)=\hat{y}(\lambda)\]
and for $\lambda\in \mathbb{D}_-$
\[\lambda\widehat{Tx}(\lambda)=\hat{x}(\lambda)=\lambda \hat{y}(\lambda).\]
These show that $\widehat{Tx-y}=0$ and consequently $y=Tx$.
\end{proof}
\begin{corollary}\label{c2}Let $\mathcal{N}:=\textup{span}\{\cup_{\lambda\in \mathbb{D}_+\cup \mathbb{D}_-}N_\lambda\}$. Then $\mathcal{N}$ is dense in $A_T^{\bot_s}$.
\end{corollary}
\begin{proof}Since $(\mathcal{N}^{\bot_s})^{\bot_s}=\overline{\mathcal{N}}$, we only have to prove that $\mathcal{N}^{\bot_s}=A_T$. Clearly, $A_T\subseteq \mathcal{N}^{\bot_s}$. Now let $(x,y)\in \mathcal{N}^{\bot_s}$. Then for $\lambda\in \mathbb{D}_+$ and each $\omega\in F^\dag_\lambda$, $(\bar{\lambda}\omega, \omega)\in N_{\bar{\lambda}}$ where $\bar{\lambda}\in \mathbb{D}_-$, and
\[(x,\bar{\lambda}\omega)_H-(y,\omega)_H=0,\]
which is equivalent to
\[((\lambda\hat{x}(\lambda)-\hat{y}(\lambda),\omega))=0.\]
Thus for each $\lambda\in \mathbb{D}_+$, $\lambda\hat{x}(\lambda)-\hat{y}(\lambda)=0$. Similarly, for each $\lambda\in \mathbb{D}_-$, $\hat{x}(\lambda)-\lambda\hat{y}(\lambda)=0$. The previous proposition leads to the claim.
\end{proof}
If a boundary quadruple $(H_\pm, \Gamma_\pm)$ for $A_T^{\bot_s}$ is chosen and $B(\lambda): \mathbb{D}_+\rightarrow \mathbb{B}(H_+, H_-)$ is the corresponding contractive Weyl function, we can obtain a localized form of $K(\lambda, \mu)$ and a concrete model for $A_T^{\bot_s}$ in terms of holomorphic vector-valued functions. Note that the associated $\gamma$-fields $\gamma_\pm$ provide a trivialization for $E$, i.e., $E|_{\mathbb{D}_+}$ can be identified with $\mathbb{D}_+ \times H_+$ via $\varphi_+:=pr_1\circ\gamma_+$ and $E|_{\mathbb{D}_-}$ with $\mathbb{D}_- \times H_-$ via $\varphi_-:=pr_2\circ\gamma_-$.
\begin{proposition}\label{p2}If $\lambda, \mu\in \mathbb{D}_+$, then for $x,y\in H_+$
\[(\varphi_+(\lambda)x, \varphi_+(\mu)y)_H=\frac{1}{1-\lambda\bar{\mu}}((Id-B(\mu)^*B(\lambda))x,y)_{H_+};\]
If $\lambda, \mu\in \mathbb{D}_-$, then for $x,y\in H_-$
\[(\varphi_-(\lambda)x, \varphi_-(\mu)y)_H=\frac{1}{1-\lambda\bar{\mu}}((Id-B(\bar{\mu})B(\bar{\lambda})^*)x,y)_{H_-};\]
If $\lambda\in \mathbb{D}_+$, $\mu\in \mathbb{D}_-$, $x\in H_+$ and $y\in H_-$, then
\[(\varphi_+(\lambda)x, \varphi_-(\mu)y)_H=\frac{1}{\lambda-\bar{\mu}}((B(\lambda)-B(\bar{\mu}))x,y)_{H_-}.\]
\end{proposition}
\begin{proof}We only prove the first statement for the others follow similarly. Note that
\[\gamma_+(\lambda)x=(\varphi_+(\lambda)x, \lambda \varphi_+(\lambda)x )\in N_\lambda.\]
Then the abstract Green formula Eq.~(\ref{gre}) applied to $\gamma_+(\lambda)x, \gamma_+(\mu)y$ results in
\begin{eqnarray*}(1-\lambda\bar{\mu})(\varphi_+(\lambda)x, \varphi_+(\mu)y)_H&=&(\Gamma_+\gamma_+(\lambda)x, \Gamma_+\gamma_+(\mu)y)_{H_+}-(\Gamma_-\gamma_+(\lambda)x, \Gamma_-\gamma_+(\mu)y)_{H_-}\\
&=&(x,y)_{H_+}-(B(\lambda)x,B(\mu)y)_{H_-}.\end{eqnarray*}
The claim then follows.
\end{proof}
\emph{Remark}. In particular, these formulae give the Hermitian structure on $E$ in terms of the contractive Weyl function $B(\lambda)$, e.g., at $\lambda\in \mathbb{D}_+$,
\[(\varphi_+(\lambda)x, \varphi_+(\lambda)y)_H=\frac{1}{1-|\lambda|^2}((Id-B(\lambda)^*B(\lambda))x,y)_{H_+}.\]

$\varphi_\pm$ also induce a trivialization of $F$: We can define $\varphi_\pm^\dag(\lambda): F_\lambda\rightarrow H_\pm$ where $\lambda\in \mathbb{D}_\mp$ in its usual way, e.g.,
\[(\varphi_+^\dag(\lambda)\omega, x)_{H_+}=((\omega, \varphi_+(\bar{\lambda})x)),\quad  \forall \lambda\in \mathbb{D}_-,\  \omega\in F_\lambda,\  x\in H_+.\]
Then $\psi_\pm(\lambda):=(\varphi_\mp^\dag(\lambda))^{-1}$ for $\lambda\in \mathbb{D}_\pm$ identify the bundle $F$ with $(\mathbb{D}_+\times H_-)\cup (\mathbb{D}_-\times H_+)$. With $\psi_\pm$, $\mathfrak{H}$ will be interpreted as a reproducing kernel Hilbert space of vector-valued holomorphic functions $f$ on $\mathbb{D}_+\cup \mathbb{D}_-$. More precisely, for $s\in \mathfrak{H}$, the corresponding function $f_s(\lambda):=(\psi_\pm(\lambda))^{-1}s(\lambda)\in H_\mp$ for $\lambda\in \mathbb{D}_\pm$. Let $\mathcal{K}(\lambda, \mu)$ be the localized form of $K(\lambda, \mu)$ w.r.t. $\varphi_\pm$ and $\psi_\pm$, e.g.,
\[\mathcal{K}(\lambda, \mu)=(\psi_+(\lambda))^{-1}\circ K(\lambda, \mu)\circ\varphi_-(\bar{\mu})\in \mathbb{B}(H_-)\]
for $\lambda, \mu\in \mathbb{D}_+$ and
\[\mathcal{K}(\lambda, \mu)=(\psi_-(\lambda))^{-1}\circ K(\lambda, \mu)\circ\varphi_-(\bar{\mu})\in \mathbb{B}(H_-,H_+)\]
for $\lambda\in \mathbb{D}_-, \mu\in \mathbb{D}_+$.
\begin{theorem}\label{t2}With the above notation, for $\lambda, \mu\in \mathbb{D}_+$,
\[\mathcal{K}(\lambda, \mu)=\frac{Id-B(\lambda)B(\mu)^*}{1-\lambda\bar{\mu}};\]
For $\lambda, \mu\in \mathbb{D}_-$,
\[\mathcal{K}(\lambda, \mu)=\frac{Id-B(\bar{\lambda})^*B(\bar{\mu})}{1-\lambda\bar{\mu}};\]
For $\lambda\in \mathbb{D}_+$ and $\mu\in \mathbb{D}_-$,
\[\mathcal{K}(\lambda, \mu)=\frac{B(\lambda)-B(\bar{\mu})}{\lambda-\bar{\mu}};\]
For $\lambda\in \mathbb{D}_-$ and $\mu\in \mathbb{D}_+$,
\[\mathcal{K}(\lambda, \mu)=\frac{B(\bar{\lambda})^*-B(\mu)^*}{\lambda-\bar{\mu}}.\]
\end{theorem}
 \begin{proof}We only check the first and the third formulae, and the others follow similarly. Let $x,y\in H_-$ and $\lambda, \mu\in \mathbb{D}_+$. Then
 \begin{eqnarray*}(\mathcal{K}(\lambda, \mu)x,y)_{H_-}&=&((\psi_+(\lambda))^{-1}\iota_\lambda^\dag\iota_\mu \varphi_-(\bar{\mu})x,y)_{H_-}=(\varphi_-^{\dag}(\lambda)\iota_\lambda^\dag\iota_\mu \varphi_-(\bar{\mu})x,y)_{H_-}\\
 &=&((\iota_\lambda^\dag\iota_\mu \varphi_-(\bar{\mu})x, \varphi_-(\bar{\lambda})y))=(\iota_\mu \varphi_-(\bar{\mu})x, \iota_\lambda\varphi_-(\bar{\lambda})y)_H\\
 &=&( \varphi_-(\bar{\mu})x, \varphi_-(\bar{\lambda})y)_H=\frac{1}{1-\lambda\bar{\mu}}((Id-B(\lambda)B(\mu)^*)x,y)_{H_-}.
 \end{eqnarray*}
 The last equality is due to Prop.~\ref{p2} and the claimed formula then follows. Similarly, let $x\in H_+$, $y\in H_-$, $\lambda\in \mathbb{D}_+$ and $\mu\in \mathbb{D}_-$. Then
 \begin{eqnarray*}(\mathcal{K}(\lambda, \mu)x,y)_{H_-}&=&((\psi_+(\lambda))^{-1}\iota_\lambda^\dag\iota_\mu \varphi_+(\bar{\mu})x,y)_{H_-}=(\varphi_-^{\dag}(\lambda)\iota_\lambda^\dag\iota_\mu \varphi_+(\bar{\mu})x,y)_{H_-}\\
 &=&((\iota_\lambda^\dag\iota_\mu \varphi_+(\bar{\mu})x, \varphi_-(\bar{\lambda})y))=( \varphi_+(\bar{\mu})x, \varphi_-(\bar{\lambda})y)_H.
 \end{eqnarray*}
 The claimed formula again follows from Prop.~\ref{p2}.
 \end{proof}
 The following theorem is of basic importance.
  \begin{theorem}\label{t1}Let $a=(x,y)\in A_T^{\bot_s}\subseteq \mathbb{H}$. Then for $\lambda\in \mathbb{D}_+$,
 \[\lambda f_{\hat{x}}(\lambda)- f_{\hat{y}}(\lambda)=B(\lambda)\Gamma_+a-\Gamma_-a,\]
  and for $\lambda\in \mathbb{D}_-$,
  \[f_{\hat{x}}(\lambda)-\lambda f_{\hat{y}}(\lambda)=\Gamma_+a-B(\bar{\lambda})^*\Gamma_-a.\]
 \end{theorem}
 \begin{proof}For $\lambda\in \mathbb{D}_+$ and $\omega\in F^\dag_\lambda$, we have $\tilde{\omega}:=(\bar{\lambda}\omega, \omega)\in N_{\bar{\lambda}}\subseteq A_T^{\bot_s}$ where $\bar{\lambda}\in \mathbb{D}_-$ and due to the abstract Green's formula Eq.~(\ref{gre})
 \begin{eqnarray*}(x,\bar{\lambda} \omega)_H-(y, \omega)_H&=&(\Gamma_+ a,\Gamma_+ \tilde{\omega})_{H_+}-(\Gamma_- a,\Gamma_- \tilde{\omega})_{H_-}\\
 &=&(\Gamma_+ a,B(\lambda)^*\Gamma_- \tilde{\omega})_{H_+}-(\Gamma_- a,\Gamma_- \tilde{\omega})_{H_-}\\
 &=&(B(\lambda)\Gamma_+ a-\Gamma_- a,\Gamma_- \tilde{\omega})_{H_-}.
 \end{eqnarray*}
 On the other side,
 \begin{eqnarray*}(x,\bar{\lambda} \omega)_H-(y, \omega)_H&=&\lambda((\iota_\lambda^\dag x, \omega))-((\iota_\lambda^\dag y, \omega))=((\lambda\hat{x}(\lambda), \omega))-((\hat{y}(\lambda), \omega))\\
 &=&((\lambda\hat{x}(\lambda)-\hat{y}(\lambda), \varphi_-(\bar{\lambda})\Gamma_-\tilde{\omega}))=(\varphi_-^\dag(\lambda)(\lambda\hat{x}(\lambda)-\hat{y}(\lambda)), \Gamma_-\tilde{\omega})_{H_-}\\
 &=&(\lambda f_{\hat{x}}(\lambda)-f_{\hat{y}}(\lambda), \Gamma_-\tilde{\omega})_{H_-}.
 \end{eqnarray*}
 Combining the above computations leads to the first formula. The other formula can be proved in a similar manner and we omit the details.
 \end{proof}
  This theorem actually characterizes $\widehat{A_T^{\bot_s}}$ completely.
  \begin{proposition}$(f,g)\in \mathfrak{H}\oplus_\bot \mathfrak{H}$ lies in $\widehat{A_T^{\bot_s}}$ if and only if there exist $x_\pm \in H_\pm$ such that
  for $\lambda\in \mathbb{D}_+$,
 \[\lambda f(\lambda)- g(\lambda)=B(\lambda)x_+-x_-,\]
  and for $\lambda\in \mathbb{D}_-$,
  \[f(\lambda)-\lambda g(\lambda)=x_+-B(\bar{\lambda})^*x_-.\]
  \end{proposition}
  \begin{proof}The necessity part is just Thm.~\ref{t1}. To prove the sufficiency part, let $S$ be the subspace of elements in $\mathfrak{H}\oplus_\bot \mathfrak{H}$ satisfying the property in the statement of the proposition. Since the evaluation map $ev_\lambda$ is continuous, it's easy to see $S$ is closed in $\mathfrak{H}\oplus_\bot \mathfrak{H}$. Then due to Coro.~\ref{c2}, we only have to prove that $\widehat{\mathcal{N}}$ is dense in $S$. The necessity part already implies $\widehat{\mathcal{N}}\subseteq S$. Now let $(f,g)\in S$ be orthogonal to $\widehat{\mathcal{N}}$. Then for $\mu\in \mathbb{D}_+$ and $\omega\in F_\mu^\dag$, by definition
  \[0=(f, \bar{\mu}f_{\hat{\omega}})_{\mathfrak{H}}+(g, f_{\hat{\omega}})_{\mathfrak{H}}=(\mu f+g, f_{\hat{\omega}})_{\mathfrak{H}}.\]
  Note that $\hat{\omega}(\cdot)=K(\cdot, \mu)\omega$. Then the reproducing property implies from the above equation that
  \[\mu f(\mu)+g(\mu)=0.\]
  In particular, $g(0_+)=0$. Similarly we can find $f(0_-)=0$. On the other side, since $(f,g)\in S$, hence there exist $x_\pm\in H_\pm$ such that
  \[B(0_+)x_+-x_-=0,\quad x_+-B(0_+)^*x_-=0.\]
  Due to the fact that $\|B(0_+)\|<1$, we ought to have $x_\pm=0$. Then by Prop.~\ref{p5}, $(f,g)\in \widehat{A_T}\subset \widehat{\overline{\mathcal{N}}}$. Thus $f=g=0$ and the proposition is proved.
  \end{proof}
  Now we are prepared to construct a functional model for $T$.
 \begin{corollary}If, in terms of the boundary quadruple $(H_\pm, \Gamma_\pm)$, $T$ as an extension of its unitary map part is characterized by $\mathcal{C}\in \mathbb{B}(H_+,H_-)$, i.e., the graph of $T$ is
 \[\{a\in A_T^{\bot_s}|\Gamma_-a=\mathcal{C}\Gamma_+a\},\]
 then
 for $\lambda\in \mathbb{D}_+$,
 \[\lambda f_{\hat{x}}(\lambda)- f_{\widehat{Tx}}(\lambda)=(B(\lambda)-\mathcal{C})\Gamma_+a,\]
  and for $\lambda\in \mathbb{D}_-$,
  \[f_{\hat{x}}(\lambda)-\lambda f_{\widehat{Tx}}(\lambda)=(Id-B(\bar{\lambda})^*\mathcal{C})\Gamma_+a\]
  where $a=(x,Tx)\in A_T^{\bot_s}$.
  \end{corollary}
 \begin{proof}This can be obtained by simply setting $\Gamma_-a=\mathcal{C}\Gamma_+a$ in Thm.~\ref{t1}.
 \end{proof}
 If we set $\mathfrak{T}f_{\hat{x}}=f_{\widehat{Tx}}$, then $\mathfrak{T}$ is obviously a model for $T$ in $\mathfrak{H}$. If furthermore $Id-B(0_+)^*\mathcal{C}$ is (possibly unboundedly) invertible, then
 \[\Gamma_+a=(Id-B(0_+)^*\mathcal{C})^{-1}f_{\hat{x}}(0_-)\]
 and consequently for any $f\in \mathfrak{H}$,
 \[(\mathfrak{T}f)(\lambda)=\lambda f(\lambda)-(B(\lambda)-\mathcal{C})(Id-B(0_+)^*\mathcal{C})^{-1}f(0_-),\quad \forall \lambda\in \mathbb{D}_+,\]
 \[(\mathfrak{T}f)(\lambda)=[f(\lambda)-(Id-B(\bar{\lambda})^*\mathcal{C})(Id-B(0_+)^*\mathcal{C})^{-1}f(0_-)]/\lambda,\quad \forall \lambda\in \mathbb{D}_-.\]
 \begin{example}If $(\mathbb{K}^\bot, \mathbb{K}^\bot_*, \Gamma_\pm)$ is the canonical boundary quadruple, then the corresponding contractive Weyl function satisfies $B(0_+)=0$ and $T$ is parameterized by $\Gamma_-a=t\Gamma_+ a$ for $a\in A_T^{\bot_s}$. Thus the model operator $\mathfrak{T}$ is given in the following way:
 \[(\mathfrak{T}f)(\lambda)=\lambda f(\lambda)-(B(\lambda)-t)f(0_-),\quad \forall\ \lambda\in \mathbb{D}_+,\]
 \[(\mathfrak{T}f)(\lambda)=[f(\lambda)-(Id-B(\bar{\lambda})^*t)f(0_-)]/\lambda,\quad \forall\ \lambda\in \mathbb{D}_-.\]
 \end{example}
 \begin{corollary}\label{c1}If, in terms of the boundary quadruple $(H_\pm, \Gamma_\pm)$, $T$ as an extension of its unitary map part is characterized by $0\in \mathbb{B}(H_+,H_-)$, then
 for any $f\in \mathfrak{H}$,
 \[(\mathfrak{T}f)(\lambda)=\lambda f(\lambda)-B(\lambda)f(0_-),\ \forall\ \lambda\in \mathbb{D}_+\]
 and
 \[(\mathfrak{T}f)(\lambda)=(f(\lambda)-f(0_-))/\lambda,\ \forall\ \lambda\in \mathbb{D}_-.\]
 \end{corollary}
 \begin{proof}Setting $\mathcal{C}=0$ in the above argument leads to the formulae.
 \end{proof}
 \begin{example}If $\|t\|<1$ and the boundary quadruple $(\mathbb{K}^\bot, \mathbb{K}^\bot_*, \Gamma_\pm')$ in Example \ref{e6} is used, then $T$ is parameterized by $\Gamma_-'a=0$ in $A_T^{\bot_s}$ and $B(\lambda)$ is hence $-\Theta_T(\lambda)$. In this specific setting, the result in Coro.~\ref{c1} is essentially the de Branges-Rovnyak model. It should be pointed out that in the de Branges-Rovnyak model, the model space $\mathfrak{H}$ is a space of two-component holomorphic functions on a single disc $\mathbb{D}$. See \cite{ball2014branges} for a detailed and up-to-date account of the de Branges-Rovnyak model, where the model space is denoted by $\mathcal{D}(S)$.
 \end{example}

\section{From marked Nevanlinna discs to contractions}\label{sec5}
It is well-known that associated to a pure-contraction valued Schur function $\Theta(\lambda)$ on the unit disc there is a c.n.u. contraction $T$ such that $\Theta(\lambda)$ is precisely the Sz.-Nagy-Foias characteristic function of $T$. In our formalism, the characteristic function is replaced by the Weyl curve. Since the Weyl curve of a c.n.u. contraction $T$ only contains information concerning the unitary map part of $T$, we can not expect that the Weyl curve determines the contraction $T$ totally.

Let $H$ be a strong symplectic Hilbert space with signature $(n_+,n_-)$ and $W_\pm(H)$ the complex manifolds defined in \S~\ref{sec2}.
\begin{definition}A Nevanlinna disc $N(\lambda)$ in $W_+(H)$ is a holomorphic map from the unit disc $\mathbb{D}$ to $W_+(H)$. If additionally $L$ is a maximal positive-definite subspace of $H$, we call the pair $(N(\lambda), L)$ a marked Nevanlinna disc.
\end{definition}
Note that if $L$ is maximally completely positive-definite, then $L$ is just an additional point in $W_+(H)$ beside $N(\lambda)$. Motivated by the construction in \S~\ref{sec3.2} and the investigation in \S~\ref{sec4}, in this section we show how a contraction can be constructed from a marked Nevanlinna disc.

Given a Nevanlinna disc $N(\lambda)$, we can extend it to a map $\mathcal{N}(\lambda)$ from $\mathbb{D}_+\cup \mathbb{D}_-$ to $W_+(H)\cup W_-(H)$ by simply setting $\mathcal{N}(\lambda)=N(\lambda)$ for $\lambda\in \mathbb{D}_+$ and $\mathcal{N}(\lambda)=N(\bar{\lambda})^{\bot_s}\in W_-(H)$ for $\lambda\in \mathbb{D}_-$. We pull back the tautological holomorphic vector bundle over $W_+(H)\cup W_-(H)$ to obtain a holomorphic Hermitian vector bundle $E$ over $\mathbb{D}_+\cup \mathbb{D}_-$. Exchanging the fiber $E_\lambda$ at $\lambda\in \mathbb{D}_+$ with the fiber $E_{\bar{\lambda}}$ for $\bar{\lambda}\in \mathbb{D}_-$ produces an anti-holomorphic vector bundle $F^\dag$. Let $F$ be the conjugate linear dual of $F^\dag$, which is holomorphic and equipped with the induced Hermitian structure.

 Let us construct a reproducing kernel on $F$. Note that at each $\lambda\in \mathbb{D}_+\cup \mathbb{D}_-$ there is the decomposition $H=F^\dag_\lambda\oplus F^\dag_{\bar{\lambda}}$, where if $\lambda$ lies in one copy of the disc, then $\bar{\lambda}$ lies in the other. Denote the projection along $F^\dag_{\bar{\lambda}}$ onto $F^\dag_\lambda$ by $\mathcal{P}_\lambda$. For $\lambda, \mu\in \mathbb{D}_+\cup \mathbb{D}_-$, we can construct a map $\mathcal{Q}(\lambda, \mu)$ from $F^\dag_\mu$ to $F_\lambda$: If $\omega\in F^\dag_\mu$,
then $\mathcal{P}_\lambda\omega\in F^\dag_\lambda$ and $(\mathcal{P}_\lambda\omega,\cdot)_\lambda$ (the Hermitian structure on the fiber is used) is a conjugate-linear functional on $F^\dag_\lambda$. Consequently, $(\mathcal{P}_\lambda\omega,\cdot)_\lambda\in F_\lambda$ and we simply set $\mathcal{Q}(\lambda, \mu)\omega=(\mathcal{P}_\lambda\omega,\cdot)_\lambda$. The kernel $K(\lambda, \mu)$ now can be defined as follows:
\begin{itemize}
  \item If $\lambda, \mu$ are in the same copy of the disc, we set $K(\lambda, \mu)=\frac{\mathcal{Q}(\lambda,\mu)}{1-\lambda\bar{\mu}}$;
  \item If $\lambda,\mu$ are in different copies of the disc, we set $K(\lambda, \mu)=-\frac{\mathcal{Q}(\lambda,\mu)}{\lambda-\bar{\mu}}$.
  \end{itemize}
In the latter case, if $\lambda=\bar{\mu}$, $\mathcal{Q}(\bar{\mu},\mu)$ is certainly zero. It turns out that $\mathcal{Q}(\lambda, \mu)$ as an analytic function in $\lambda$ has a zero at $\bar{\mu}$, and we can simply set $K(\bar{\mu},\mu)$ to be the limit of the above formula as $\lambda$ approaches $\bar{\mu}$.

To see $K(\lambda, \mu)$ is a sesqui-analytic, locally uniformly bounded and positive kernel, we can choose a polarization of $H$ to obtain a localized form of $K(\lambda, \mu)$ and check these properties routinely. Without loss of generality, we can choose the polarization $H_+:=N(0)=\mathcal{N}(0_+)$ and identify $H$ with the standard strong symplectic Hilbert space $H_+\oplus_\bot H_-$ in Example \ref{e3} where $H_-=H_+^{\bot_s}$. In this way, the Nevanlinna disc has the following representation:
\[N(\lambda)=\{(x,B(\lambda)x)\in H_+\oplus_\bot H_-|x\in H_+\},\]
where $B(\lambda)$ is a $\mathbb{B}(H_+,H_-)$-valued Schur function such that $\|B(\lambda)\|<1$. As a consequence,
\[N(\lambda)^{\bot_s}=\{(B(\lambda)^\ast x,x)\in H_+\oplus_\bot H_-|x\in H_-\}.\]

We set $\gamma_+(\lambda)x=(x,B(\lambda)x)\in E_\lambda$ for $\lambda\in \mathbb{D}_+$, $x\in H_+$ and $\gamma_-(\lambda)x=(B(\bar{\lambda})^\ast x,x)\in E_\lambda$ for $\lambda\in \mathbb{D}_-$, $x\in H_-$. These provide a trivialization for $E$ and $F^\dag$. In particular, we find that on $E_\lambda$ for $\lambda\in \mathbb{D}_+$ the Hermitian structure is given by
\[(\gamma_+(\lambda)x, \gamma_+(\lambda)y)_\lambda=((Id-B(\lambda)^*B(\lambda))x,y)_+,\quad x,y\in H_+.\]
Similarly, on $E_\lambda$ for $\lambda\in \mathbb{D}_-$,
\[(\gamma_-(\lambda)x, \gamma_-(\lambda)y)_\lambda=((Id-B(\bar{\lambda})B(\bar{\lambda})^*)x,y)_-,\quad x,y\in H_-.\]
Additionally, we define $\varphi_\mp^\dag(\lambda): F_\lambda\rightarrow H_\mp$ for $\lambda\in \mathbb{D}_\pm$ via
\[(\varphi_\mp^\dag(\lambda)\omega, x)_\mp=((\omega, \gamma_\mp(\bar{\lambda})x)),\quad \forall\ \omega\in F_\lambda, \ x\in H_\mp.\]
Set $\psi_\pm(\lambda):=(\varphi_\mp^\dag(\lambda))^{-1}$. They can be used to trivialize $F$. Let $\mathcal{K}(\lambda, \mu)$ be the localized form of $K(\lambda,\mu)$ w.r.t. $\gamma_\pm$ and $\psi_\pm$.
\begin{proposition}In terms of the above $\gamma_\pm$ and $\psi_\pm$, $\mathcal{K}(\lambda, \mu)$ takes the same form as in Thm.~\ref{t2}.
\end{proposition}
\begin{proof}We only consider the cases (i) $\lambda, \mu\in \mathbb{D}_+$ and (ii) $\lambda\in \mathbb{D}_-, \mu\in \mathbb{D}_+$. Other cases can be checked similarly.

(i) Let $x\in H_-$. Then $\gamma_-(\bar{\mu})x=(B(\mu)^*x, x)$. According to the decomposition $H=F_\lambda^\dag \oplus F_{\bar{\lambda}}^\dag$, we have
\[(B(\mu)^*x, x)=(B(\lambda)^*y_-, y_-)+(y_+, B(\lambda)y_+)\]
for some $y_\pm\in H_\pm$, i.e.,
\[B(\mu)^*x=B(\lambda)^*y_-+y_+,\quad x=y_-+B(\lambda)y_+.\]
We can get
\begin{equation}y_-=(Id-B(\lambda)B(\lambda)^*)^{-1}(Id-B(\lambda)B(\mu)^*)x.\label{e9}\end{equation}
Note that by definition $\mathcal{P}_\lambda\gamma_-(\bar{\mu})x=\gamma_-(\bar{\lambda})y_-$. Therefore for any $w\in H_-$,
\[(\mathcal{P}_\lambda\gamma_-(\bar{\mu})x, \gamma_-(\bar{\lambda})w)_\lambda=(\gamma_-(\bar{\lambda})y_-, \gamma_-(\bar{\lambda})w)_\lambda=((Id-B(\lambda)B(\lambda)^*)y_-,w)_-,\]
where $(\cdot, \cdot)_\lambda$ is the inner product on $F^\dag_\lambda$. Due to Eq.~(\ref{e9}), we have
\[((\mathcal{Q}(\lambda, \mu)\gamma_-(\bar{\mu})x,\gamma_-(\bar{\lambda})w))=(\mathcal{P}_\lambda\gamma_-(\bar{\mu})x, \gamma_-(\bar{\lambda})w)_\lambda=((Id-B(\lambda)B(\mu)^*)x,w)_-.\]
This certainly means
\[\mathcal{K}(\lambda, \mu)=\frac{Id-B(\lambda)B(\mu)^*}{1-\lambda \bar{\mu}}.\]
(ii) Let $x\in H_-$ and $\gamma_-(\bar{\mu})x=(B(\mu)^*x, x)$. Like in (i), now we have
\[(B(\mu)^*x, x)=(y_+, B(\bar{\lambda})y_+)+(B(\bar{\lambda})^*y_-, y_-)\]
for some $y_\pm\in H_\pm$. We can obtain
\[y_+=(Id-B(\bar{\lambda})^*B(\bar{\lambda}))^{-1}(B(\mu)^*-B(\bar{\lambda})^*)x.\]
By definition, now $\mathcal{P}_\lambda\gamma_-(\bar{\mu})x=\gamma_+(\bar{\lambda})y_+$. For $w\in H_+$,
\[(\mathcal{P}_\lambda\gamma_-(\bar{\mu})x, \gamma_+(\bar{\lambda})w)_\lambda=(\gamma_+(\bar{\lambda})y_+,\gamma_+(\bar{\lambda})w)_\lambda=((B(\mu)^*-B(\bar{\lambda})^*)x,w)_+,\]
implying in this case
\[\mathcal{K}(\lambda,\mu)=\frac{B(\bar{\lambda})^*-B(\mu)^*}{\lambda-\bar{\mu}}.\]
\end{proof}
Let $\mathfrak{H}$ be the reproducing kernel Hilbert space generated by $K(\lambda, \mu)$ and $\mathbb{H}:=\mathfrak{H}\oplus_\bot \mathfrak{H}$ equipped with the standard strong symplectic structure in Example \ref{e3} where $H_+=H_-=\mathfrak{H}$. Define
\[\mathfrak{A}=\{(f,g)\in \mathbb{H}|g(\lambda)=\lambda f(\lambda),\ \forall \lambda\in \mathbb{D}_+ \ \textup{and}\ f(\lambda)=\lambda g(\lambda),\ \forall \lambda\in \mathbb{D}_-\}.\]
Two Nevanlinna discs in $W_+(H_1)$ and $W_+(H_2)$ respectively are called congruent if there is a symplectic isomorphism between $H_1$ and $H_2$, turning one Nevanlinna disc into the other.
\begin{proposition}$\mathfrak{A}$ is isotropic in $\mathbb{H}$, whose Weyl curve is congruent to $N(\lambda)$.
\end{proposition}
\begin{proof}As the proof is similar to that of Thm.~5.20 in \cite{wang2024complex}, we only outline the proof. We use $\check{H}$ to denote the same Hilbert space $H$ but with the opposite strong symplectic structure $-[\cdot, \cdot]$. As in Example \ref{e10}, we obtain the direct sum $\mathbb{H}\oplus_\bot \check{H}$. In the following, we shall still use the identification $H=H_+\oplus_\bot H_-$ and consider $\mathfrak{H}$ as a space of holomorphic functions on $\mathbb{D}_+\cup \mathbb{D}_-$. In $\mathbb{H}\oplus_\bot \check{H}$, we can define a subspace
\[\mathbb{L}:=\{(f,g;x_+,x_-)\in \mathbb{H}\oplus_\bot \check{H}| \textup{for} \, \lambda\in \mathbb{D}_+,\, g(\lambda)=\lambda f(\lambda)-(B(\lambda)x_+-x_-),\]
\[\, \textup{for} \, \lambda\in \mathbb{D}_-,\, f(\lambda)=\lambda g(\lambda)+(x_+-B(\bar{\lambda})^*x_-)\}.\]
One can adapt the proof of Prop.~5.21 in \cite{wang2024complex} to show that $\mathbb{L}$ is Lagrangian in $\mathbb{H}\oplus_\bot \check{H}$. $\mathfrak{A}$ can be embedded in $\mathbb{L}$ by mapping $(f,g)\in \mathfrak{A}$ to $(f,g;0,0)\in \mathbb{L}$. This is enough to show $\mathfrak{A}$ is isotropic.

$\mathfrak{A}^{\bot_s}$ can also be characterized in terms of $\mathbb{L}$:
\[\mathfrak{A}^{\bot_s}=\{(f,g)\in \mathbb{H}| \exists x_\pm\in H_\pm,\ s.t.\ (f,g; x_+, x_-)\in \mathbb{L}\}.\]
This can be proved in the same way Prop.~5.24 in \cite{wang2024complex} was proved. For $a=(f,g)\in \mathfrak{A}^{\bot_s}$, the above $x_\pm\in H_\pm$ actually are uniquely determined. Set $\Gamma_\pm a=x_\pm$. Then $(H_\pm, \Gamma_\pm)$ is a boundary quadruple for $\mathfrak{A}^{\bot_s}$ and the corresponding contractive Weyl function is precisely $B(\lambda)$.
\end{proof}
We know very well now how to associate a contraction $\mathfrak{T}$ to a marked Nevanlinna disc $(N(\lambda),L)$. One can use $N(\lambda)$ to define a reproducing kernel Hilbert space $\mathfrak{H}$ in the above manner and obtain $\mathfrak{A}^{\bot_s}$. To obtain $\mathfrak{T}$ amounts to cutting out a suitable subspace of $\mathfrak{A}^{\bot_s}$. If as before the polarization $N(0)$ in $H$ is chosen and hence $L$ is parameterized by $\mathcal{C}\in \mathbb{B}(H_+,H_-)$, then the graph of $\mathfrak{T}$ is
\[\{a\in \mathfrak{A}^{\bot_s}|\Gamma_-a=\mathcal{C}\Gamma_+a\}.\]
In particular, given a c.n.u. contraction $T$, we have a corresponding marked Nevanlinna disc $(N(\lambda), L)$ in $W_+(\mathcal{B}_T)$ where $N(\lambda)$ is the Weyl curve of $T$ and $L$ is the maximal positive-definite subspace in $\mathcal{B}_T$ determined by the non-unitary map part of $T$. Then the above construction gives an alternative way to produce the functional model for $T$ in \S~\ref{sec4}.

Two marked Nevanlinna discs $(N_1(\lambda), L_1)$ in $W_+(H_1)$ and $(N_2(\lambda), L_2)$ in $W_+(H_2)$ are called congruent if there is a symplectic isomorphism $\Phi$ between $H_1$ and $H_2$ such that $\Phi$ turns $N_1(\lambda)$ into $N_2(\lambda)$ and $L_1$ into $L_2$. The following proposition can be obtained without much effort.
\begin{proposition}Two c.n.u. contractions are unitarily equivalent if and only if their corresponding marked Nevanlinna discs are congruent.
\end{proposition}
\begin{proof}The necessity part is clear. For two c.n.u. contractions $T_1$ and $T_2$, if their corresponding marked Weyl curve are congruent, we can fix a standard strong symplectic Hilbert space $H=H_+\oplus_\bot H_-$ as in Example \ref{e3} and choose suitable symplectic isomorphisms $\Phi_1: \mathcal{B}_{T_1}\rightarrow H$ and $\Phi_2: \mathcal{B}_{T_2}\rightarrow H$ to map the two marked Nevanlinna discs to the same marked Nevanlinna disc $(N(\lambda),L)$ in $W_+(H)$. Then by the above argument the model operator $\mathfrak{T}$ associated to $(N(\lambda),L)$ is unitarily equivalent to both $T_1$ and $T_2$. As a consequence, $T_1$ and $T_2$ are unitarily equivalent as well.
\end{proof}
Recall that closely related to the characteristic function $\Theta_T(\lambda)$ is the invariant subspace problem. Actually, according to Douglas in \cite{douglas1974canonical}, "Much of the original interest in the theory of canonical models stemmed from the belief that it might play the
decisive role in solving this problem." It is known that certain factorizations of $\Theta_T(\lambda)$ correspond to invariant subspaces of $T$. It is interesting to know how to interpret this in terms of the marked Nevanlinna disc--the geometric counterpart of $\Theta_T(\lambda)$. We shall turn to this topic elsewhere in the futhure.
\bibliographystyle{alpha}

\end{document}